\documentclass[a4paper,10pt]{article}

\usepackage[utf8]{inputenc}
\usepackage[utf8]{inputenc}
\usepackage{amsmath}
\usepackage{amsfonts}
\usepackage{amssymb,amsthm}
\usepackage{todonotes} 
\usepackage{algorithm}
\usepackage{algorithmic}
\usepackage{hyperref}
\usepackage{tikz}
\usepackage[space,noadjust,nocompress]{cite}

\newtheorem{theorem}{Theorem}
\newtheorem{proposition}{Proposition}

\newtheorem{remark}{Remark}
\newtheorem{definition}{Definition}
\newtheorem{notation}{Notation}
\newtheorem{lemma}{Lemma}

\newcommand{\yd}{y_\delta} 
\newcommand{\ydi}{y_{1,\delta}} 
\newcommand{\ydii}{y_{2,\delta}}

\newcommand{\N}{\mathbb{N}}

\newcommand{\xd}{x^\dagger} 
\newcommand{\KS}{\mathcal{K}} 
\newcommand{\RS}{\mathcal{R}}
\newcommand{\KRS}{\mathcal{KR}}
\newcommand{\X}{\mathcal{X}}
\newcommand{\Q}{\mathcal{Q}} 
\newcommand{\be}{\begin{equation}}
\newcommand{\ee}{\end{equation}}
\newcommand{\ca}{c_{it}}
\newcommand{\cnu}[1]{c_{#1}}
\newcommand{\cnuit}[1]{\tilde{c}_{#1}}
\newcommand{\cgen}{\cnu{\mu_*}}
\newcommand{\we}[1]{{w}_{#1}}
\newcommand{\wek}[1]{{\hat{w}}_{#1}}
\newcommand{\pin}{\pi}
\newcommand{\pink}{\hat{\pi}}

\newcommand{\PS}{\mathcal{P}}
\newcommand{\spa}{{\rm span}} 
\newcommand{\err}{\text{err}} 

\newcommand{\Op}{A^*A}
\newcommand{\y}{A^*\yd} 
\newcommand{\xRSd}[1]{x_{\RS,#1}^\delta} 
 
\newcommand{\xKRSd}[1]{x_{\KRS,#1}^\delta} 

\newcommand{\xKSd}[1]{x_{\KS,#1}^\delta}

\newcommand{\scp}[1]{\left\langle  #1 \right\rangle}
\newcommand{\muone}[1]{\theta_{#1}}
\newcommand{\muoneti}[1]{\tilde{\theta}_{#1}}
\newcommand{\muonek}[1]{\hat{\theta}_{#1}}

\newcommand{\ha}{g}

\DeclareMathOperator*{\argmin}{argmin}
\newcommand{\skp}[1]{\langle #1 \rangle } 

\newcommand{\qcg}[2]{{\tt q}^{[#2]}_{#1}}
\newcommand{\pcg}[2]{{\tt r}^{[#2]}_{#1}}

\newcommand{\rooti}[2]{{\lambda_{1,#1}^{[#2]}}}
\newcommand{\rootik}[2]{{{\hat{\lambda}}_{1,#1}^{[#2]}}}
\newcommand{\rootX}[3]{{\lambda_{#3,#1}^{[#2]}}}
\newcommand{\rootXpure}[3]{{\lambda_{#3 #1}^{#2}}}
\newcommand{\rootXk}[3]{{{\hat{\lambda}}_{#3,#1}^{[#2]}}}
\newcommand{\rootXkpure}[3]{{{\hat{\lambda}}_{#3 #1}^{#2}}}

\newcommand{\tirootX}[3]{\tilde{\lambda}_{#3,#1}^{[#2]}}

\newcommand{\pf}[2]{{r^{[#2]}_{#1}}}
\newcommand{\pfprime}[2]{{r^{[#2]'}_{#1}}}
\newcommand{\pfk}[2]{{{{\hat{r}}}_{#1}^{[#2]}}}
\newcommand{\pfkprime}[2]{{{{\hat{r}}}_{#1}^{[#2]'}}}
\newcommand{\pftau}[2]{{\tilde{r}^{[#2]}_{#1}}}

\newcommand{\rR}[1]{r_{\RS}^{#1}}
\newcommand{\rKR}[1]{r_{\KRS}^{#1}}

\newcommand{\sigk}{\hat{\sigma}} 
\newcommand{\xit}[1]{x_{{\rm IT},#1}}
\newcommand{\xiti}[1]{x_{{\rm IT},1,#1}}
\newcommand{\xitii}[1]{x_{{\rm IT},2,#1}}

\newcommand{\g}[1]{\rho^{[#1]}} 
\newcommand{\gk}[1]{\hat{\rho}^{[#1]}} 
\newcommand{\gtau}[1]{\tilde{\rho}^{[#1]}} 
\newcommand{\ya}[1]{{y^{[#1]}}} 
\newcommand{\yai}[1]{{y_{1}^{[#1]}}} 
\newcommand{\yaii}[1]{{y_{2}^{[#1]}}} 
\newcommand{\betan}[1]{\beta^{[#1]}}
\newcommand{\betank}[1]{{\hat{\beta}}^{[#1]}}
\newcommand{\betatau}[1]{{\tilde{\beta}}^{[#1]}}

\renewcommand{\P}{\mathcal{P}}
\newcommand{\phio}[2]{{\phi_{#1}^{[#2]}}}
\renewcommand{\be}{\begin{equation}} 
\renewcommand{\ee}{\end{equation}}
\newcommand{\mus}{\mu_*}
\newcommand{\xad}[1]{x_{#1}^\delta}

\title{Regularization of linear inverse problems by rational Krylov methods}
\author{Stefan Kindermann\footnote{Industrial Mathematics Institute, Johannes Kepler University, Linz.}}
\date{}
\begin{document}
\maketitle
\begin{abstract}
For approximately solving linear ill-posed problems in Hilbert spaces, we investigate the regularization 
properties of the aggregation method and the RatCG method. These recent algorithms use previously calculated 
solutions of Tikhonov regularization (respectively, Landweber iterations) 
to set up a new search space on which the least-squares 
functional is minimized. We outline how these methods can be understood as rational Krylov space methods, 
i.e., based on the space of rational functions of the forward operator. The main result is that these methods form an
optimal-order regularization schemes when combined with the discrepancy principle as stopping rule and when the underlying 
regularization parameters are sufficiently large. 
\end{abstract}

\section{Introduction}
Consider linear inverse problems given by a bounded forward operator\linebreak  
\mbox{$A: X \to Y$} acting between Hilbert spaces. Solving an 
ill-posed operator
 equation 
\begin{equation}\label{main}  A x = y^\delta,   
\end{equation}
with $y^\delta$ representing noisy data, 
usually requires a regularization 
method. One of the most prominent ones for this task is 
Tikhonov regularization, where an approximate solution to \eqref{main}
is computed by 
\begin{equation}\label{defTik}
 \xad{\alpha}:= (A^*A  + \alpha I )^{-1}A^*y^\delta, 
\end{equation} 
with a selected regularization parameter $\alpha>0$. An elaborated theory 
regarding convergence and convergence rates of this method has been established (see, e.g., \cite{EHN}). 
A drawback of the method is the need to  solve a linear system, thus, 
especially in high-dimensional cases, iterative methods are nowadays the state of the art.
 For instance, highly popular are Krylov-space methods (see, e.g.,~\cite{Regtools}) 
 such as the conjugate gradient (CG) method for the normal equations
(CGNE)~\cite{cg,EHN,Hanke}. 

However, in this article, we are focusing on a recent new modification of the classical Tikhonov method~\eqref{defTik},
namely the so-called {\em aggregation method}~\cite{ChenPere15}, which improves Tikhonov regularization 
by constructing linear combinations of several $\xad{\alpha_i}$, defined by  \eqref{defTik}, 
and minimizing the least-squares functional $\|A x - y^\delta\|^2$  over such combinations. Note that in practice, 
for the  search for a reasonable regularization parameter $\alpha$, several solutions  $\xad{\alpha_i}$ 
have to be computed anyway such that there are a couple of Tikhonov-regularized solutions $(\xad{\alpha_i})_{i=1}^N $,
 $\alpha_i > 0$,  at our disposal. 
So it suggests itself to make use of them and build a new approximate solution via 
minimizing the least-squares functional over the linear combination  
$x = \sum_{i=1}^n c_i \xad{\alpha_i}$.
%
This amounts to  a (usually low-dimensional)  least-squares problem for the coefficients $c_i$, 
which is comparably easy to solve. We discuss more details in the next section. 
Also in the next section, we  define a variant of this idea, the 
RatCG method, introduced in~\cite{KiZell},  
where mixed linear combinations
of Tikhonov-regularized solution and Landweber iteration steps  are used.

In this article we are focusing on the regularization properties of these   procedures. Note that, as we will see, 
they all represent  {\em nonlinear} methods in the data and cannot be easily handled by a standard spectral filter
approach~\cite{EHN}. 
More in detail, we will answer several questions that seem of interest: 
\begin{itemize}
 \item Are these methods  regularization methods at all? 
 \item How to choose the multiple parameters $\alpha_i$? 
 \item Do the methods inherit the saturation effect of Tikhonov regularization? 
\end{itemize}
We will address these questions in the following and in particular show 
optimal-order convergence rates  for all H\"older-type source conditions 
when the $\alpha_i$ are selected large enough
and when the discrepancy principle is used as stopping criterion. 
See the main result, Theorem~\ref{theorem1}.

 The main tool that we employ is to frame the presented algorithms
as least-squares methods in  rational Krylov spaces. We consider the 
methods as ``diagonal'' sequence of a family of CG-type methods with variable 
weights. This allows us to 
 make extensive use
of  results from orthogonal 
polynomials and the corresponding analysis in the regularization theory for  the CGNE method,
which is outlined in  the monograph of M.~Hanke~\cite{Hanke} as well as  in~\cite[Chapter 7]{EHN}. 
We also follow the main proof structure in these works, but the use of 
variable weights require nontrivial additional effort. 

Regarding prior work, we are not aware of comparable results on the regularization theory for 
rational Krylov spaces, with the exception of a recent interesting article by V.~Grimm \cite{Grimm}.
There, a complete convergence theory similar to ours is established, with the important 
difference that the regularization parameters $\alpha$ are kept constant all the time.  
Grimm's viewpoint of an rational method is hence that of defining yet another regularization method. 
In contrast, we motivate our method as a postprocessing tool for Tikhonov regularization, where 
naturally solutions with different $\alpha_i$ are available and are made use of. This may justify 
the additional effort of computing inverses. Our results include those of~\cite{Grimm} as a 
particular case, but as will be seen, the use of variable 
regularization parameter substantially increases the difficult of the proofs compared to  \cite{Grimm} 
since we have to 
estimate the difference between CG iterates for inner products with respect to different measures.

\section{Preliminaries}
Let us first introduce some standard notation and then present the methods we are investigating. 
For $A$ as in the introduction, we assume that we have an exact minimum-norm solution $\xd$ 
that solves $A \xd =: y$, with $y\in N(A^*)^\bot$ representing  unavailable "exact" data  
and with  given noisy data  $\yd = y + \err$, where $\err$ is an additive error term.
Without loss of generality we may assume that $y_\delta \in N(A^*)^\bot$, 
and this defines 
the noiselevel $\delta:= \|\yd -y\| = \|\err\|$. 
(In case of  $y_\delta \not \in N(A^*)^\bot$, we define $\delta$ as the norm of 
the projection of $\yd -y$ onto $N(A^*)^\bot$.) 
As it is usual in the (deterministic) regularization theory, 
we study  convergence and convergence rates of the
distance  of the  approximate solutions to $\xd$ as $\delta \to 0$. 

We now detail the methods of interest: 
We assume throughout a sequence of positive pairwise disjoint regularization parameters: 
\begin{align}\label{regpar}  
\alpha_1, \ldots \alpha_n,\ldots \qquad \alpha_i >0,  \qquad \alpha_i \not = \alpha_j, \quad \text{for } i\not = j. 
\end{align} 
For each $\alpha_i$, we define the associate Tikhonov solution $\xad{\alpha_i}$ by \eqref{defTik}. 
This allows us to define two new approximation spaces, $\RS^{n}$  and $ \KRS^n$,
by taking certain linear combinations of $\xad{\alpha_i}$.

For later reference and for comparison we also include the classical Krylov space $\KS^n$:  
For $n\geq1$, define the three (rational) Krylov spaces 
\begin{align*} 
%
 \KS^n &:= \spa\{\y,(\Op) \y, (\Op)^2 \y, \ldots, (\Op)^{n-1} \y \},   \\
  \RS^{n} &=  \spa\{ \xad{\alpha_1}, \xad{\alpha_2}, \ldots, \xad{\alpha_n} \}  \\
   \KRS^n  &:= \begin{cases} \RS^k \cup \KS^k &  n = 2 k, \\ 
             \RS^k \cup \KS^{k+1} &  n = 2 k +1. \\ 
            \end{cases} \\
           & = \{\y,\xad{\alpha_1}, (\Op) \y, \xad{\alpha_2}, (\Op)^2 \y,\ldots \}
\end{align*}
The Krylov space  $\KS^n$ is the basis for classical conjugate gradient  methods for 
solving the normal equations. In case of ill-posed problems, the best-known instance is 
the CGNE method \cite{Hanke}. Here, given that no degeneracy occurs, the $n$th iterate of the CGNE method  
is defined as 
\begin{equation}\label{eq:xks} \xKSd{n}:= \argmin_{x \in \KS^{n}} \|  A x - \yd\|.  \end{equation}
Degeneracy means that the Krylov space $\KS^n$ has a dimension less than $n$. If we rule out 
this case, then $\xKSd{n}$ exists and is unique. If the degeneracy condition is violated, then a 
solution $\xKSd{n}$ still exists but might not be unique. As in \cite{Hanke,KiZell}, 
we call this the break-down index:
\begin{definition}\label{def:breakdown}
Let $\X$ be  any of the spaces $\X \in \{\KS^n, \RS^n,\KRS^n\}$.  We say that the respective space $\X$ 
does not break 
down at step $n$ if 
\[ \dim \X = n. \]
We define $n_{bd,\X}$ as the smallest iteration number $n$,  where  the respective Krylov spaces $\X \in \{\KS^n, \RS^n,\KRS^n\}$
breaks down, i.e., $n_{bd,\X}-1 = \dim \X < n_{bd,\X}$.
\end{definition}
The criterion for breakdown is well-known in the Krylov space case and can be extend to the rational cases.  
The next proposition follows from \cite[Section~4]{Ruhe}; see, also \cite[Eq.~(2.1)]{BerGuettel2015}
or \cite[Prop. 2.6]{KiZell}.
\begin{proposition}\label{propbreak} 
With the previous definitions, for any $\X \in \{\KS^n, \RS^n,\KRS^n\}$, 
it holds that $n = n_{bd,\X}$ if and only if $\y$ can be written as a linear combination of $n-1$ nonzero eigenvectors of $\Op$. 
In particular 
\[ n_{bd}:= n_{bd,\KS^n} =  n_{bd,\RS^n} =   n_{bd,\KRS^n}. \]
\end{proposition}

Now in analogy to the CGNE iterations, we define two new methods by replacing the Krylov space $\KS^{n}$ 
by  $\RS^{n}$ or  $\KRS^n$.  

\begin{definition}\label{def:two}
Suppose that $\RS^{n}$ does not break down at $n$. Then the 
{\em aggregation method}
defines a unique approximate solution to \eqref{main} by 
\begin{equation}\label{eq:xrs} \xRSd{n}:= \argmin_{x \in \RS^{n}} \|  A x - \yd\|^2,  \qquad n\geq 1.  \end{equation}
Similarly, assume that $\KRS^{n}$ does not break down at $n$. Then 
the {\em  RatCG method} defines a unique approximate solution to \eqref{main} by 
\begin{equation}\label{eq:xkrs} \xKRSd{n}:= \argmin_{x \in \KRS^{n}} \|  A x - \yd\|,   \qquad n\geq 1.    \end{equation}
\end{definition}
The aggregation methods was outlined in \cite{ChenPere15} and, among others, has been successfully used in 
learning theory (see, e.g.~\cite{LearnWerner, bookSergei} and the references in \cite{KiZell}).
The RatCG method was defined in \cite{KiZell} as an iterative generalization of the aggregation method.

Regarding the algorithm, the aggregation method is implemented by 
computing the solution of a low-dimensional least-squares problem:
Assume that $n$ Tikhonov solutions have been computed. Let $x \in \RS$ be represented as 
$x = \sum_{i=1}^n c_i \xad{\alpha_i}$ with some unknown coefficients $c_i$.  
The least-squares problem in \eqref{eq:xrs} can be solved by defining the $n\times n$ matrix and right-hand side 
\[ G_{i,j}:= \skp{A \xad{\alpha_i},A \xad{\alpha_j}}, \qquad z_i = \skp{\xad{\alpha_j},A^*\yd}, \]
and then computing the coefficients 
\[ (c_i)_{i=1}^n := G^{-1} z, \] 
resulting in $\xRSd{n} = \sum_{i=1}^n c_i \xad{\alpha_i}$.
Usually, the number $n$ of Tikhonov regularized solutions $\xad{\alpha_i}$, $i = 1,\ldots n$, 
is rather low (e.g., less than $10$), so that the linear system with $G$ has low computational complexity. 

The reason for introducing the RatCG method \cite{KiZell} is  a drawback of the aggregation method: 
note that $\xRSd{n}$ is not iteratively computable because  in each step, we have to set up 
a new  matrix $G$ and invert it, and information about previous iterates can hardly be used. 
This is in contrast to the RatCG method: It can be 
defined by a short recursion, where in the even steps, a Tikhonov regularization-type of problem 
has to be solved. Details of the algorithm are given in \cite{KiZell}.
We do not discuss this iteration further but  focus here only on its
regularization properties.

Regarding the regularization aspect,  
we consider the dimension $n$ as 
main ``regularization'' parameter, although $\xRSd{n}$ and $\xKRSd{n}$ also  depends on  the choice of $\alpha_i$ as well.
For the analysis, we  are mainly interested in the convergence $\xRSd{n} \to \xd$  with $n$ chosen by a 
parameter choice rule. It turns out that the additional ``regularization'' parameters 
$\alpha_i$ are not subject to an active parameter choice but only have to be sufficiently large.
(They might even be chosen as approaching $\alpha_n \to \infty$, in strong contrast to classical regularization methods).

We  justify the proposed method, and hence the analysis in this paper, by pointing out that 
when using Tikhonov regularization, usually several solutions $\xad{\alpha_i}$ are computed anyways, 
so it would be a waste to not used them for further computations. Although Tikhonov regularization 
usually  cannot compete with the CGNE method in terms of computational complexity, it is still 
a useful and very flexible tool, as many recent publications can underline. The use of aggregation 
is definitely an improvement over pure Tikhonov regularization  bearing even more flexibility, e.g.,  
in terms of regularization parameter choices. However, our main apology refers to the fact that 
these methods can straightforwardly be generalized to nonlinear problems,  and 
this is where significant improvements over classical methods can be expected. It is our 
belief that prior to an analysis of a nonlinear situation, one should understand the 
linear case, and this is what this article can contribute.

\subsection{Representation} 
The above defined spaces can be treated in a common framework as rational Krylov spaces. 
While usual Krylov space are defined as the set of polynomial functions of $A$ or $A^*A$
applied to some right-hand side, the rational variants use rational function in place of 
polynomials.  
Such spaces are well-established for well-posed (e.g. PDE-) problems but have 
rarely been used for the regularization of ill-posed ones. We refer to the literature 
in \cite{KiZell} for further background on rational Krylov  spaces. 

Denote by $\PS^n$ the space of all polynomials of degree less than or equal to $n$. 
\[ \PS^{n} = \{p(x)\,|\, p(x) = \sum_{i=0}^n a_i x^i \}. \]
Let  $\PS_{1}^n$ be  the subspace of $\PS^{n}$  
with lowest order term $1$:
\[ \PS_1^{n} = \{p(x) \in \PS^{n},|\, p(x) = 1 \}. \] 
We also denote by 
$ \lfloor . \rfloor$ the floor function, i.e., the rounding to the next smaller integer. 

Having presented $\KS^n$, $\RS^n$ $\KRS^n$, we define the associated 
residual spaces:
\begin{equation}\label{eq:resspace} \Q_{\mathcal{X}} := \spa\{\yd-A x| x \in \mathcal{X}  \}, 
 \qquad \text{ for each } \mathcal{X} \in \{ \KS^n, \RS^n, \KRS^n \}. 
\end{equation} 
The next representation result is central for the analysis: It shows that the 
above defined spaces can all be treated as rational Krylov spaces
(see \cite[p.~331]{KiZell}):
\begin{proposition}\label{pro_1}
The Krylov spaces defined above have the following representation: 
\begin{alignat*}{3}  
\KS^n &= p_{n-1}(\Op) \y, & & & \quad &p_{n-1} \in \PS^{n-1},  \nonumber \\
\RS^n & = r_{n-1} (\Op) \y, & \qquad &r_n(x) =  \frac{p_{n-1}(x)}{\prod_{i=1}^{n} (x + \alpha_i) },&  \qquad 
&p_{n-1} \in \PS^{n-1},  \nonumber 
\\
\KRS^n & = s_{n-1}(\Op) \y,   &
&s_{n-1}(x) = \frac{p_{n-1}(x)}{\prod_{i=1}^{k} (x + \alpha_i) },&  \qquad 
&
\begin{array}{l} 
p_{n-1} \in \PS^{n-1},  \\[1ex]
k = \lfloor \frac{n}{2} \rfloor\, .
\end{array}
\end{alignat*}
Let $\Q_{\KS^n}$. $\Q_{\RS^n}$, $\Q_{\KRS^n}$ be the corresponding residual spaces defined in \eqref{eq:resspace}. Then 
\begin{alignat*}{3}  
\Q_{\KS^n}  &= p_{n}(A A^*) \yd, & & & &p_{n} \in \PS_1^{n},  \nonumber \\
\Q_{\RS^n} & = r_{n} (A A^*) \yd,& \qquad &r_n(x) =  \frac{p_{n}(x)}{\prod_{i=1}^{n} (\frac{x}{\alpha_i} + 1) },&  \qquad 
&p_{n} \in \PS_1^{n}, 
\\
\Q_{\KRS^n} & =  s_n(A A^*) \yd,& \quad 
&s_n(x) = \frac{p_n}{\prod_{i=1}^{k} (\frac{x}{\alpha_i} + 1) },& \qquad  
&\begin{array}{l} p_{n} \in \PS_1^{n}, \\[1ex]
k = \lfloor \frac{n}{2} \rfloor\, . 
\end{array}
\end{alignat*}
\end{proposition}
A similar representation holds for the residual of the normal equation, 
i.e., for $r_{NE}:= A^*\yd - \Op x$, $x \in  \KS^{n}$, $\RS^n$, or $\KRS^n$, we 
have $r_{NE} = p_{n}(\Op)\y$, $r_{NE}= r_n(\Op)\y$, or $r_{NE}= s_n(\Op)\y$, respectively, 
with the same rational functions as in Proposition~\ref{pro_1}. 
It should be kept in mind that these spaces
depend on the right-hand side $\yd$ although we do not indicate this in our notation. 

A direct consequence of the representation and the fact that 
 $\frac{1}{\frac{x}{\alpha_i} + 1} \leq 1$
is the following: 
\begin{lemma}\label{lemma:mono}
The  residuals  $\|A\xRSd{n} -\yd\|$,     $\|A\xKRSd{n} -\yd\|$,  $\|A\xKSd{n} -\yd\|$
are monotonically decreasing in $n$, and we have 
\[ \|A\xRSd{n} -\yd\| \leq \|A\xKRSd{n} -\yd\| \leq \|A\xKSd{n} -\yd\|, \]
when the same sequence of $\alpha_i$ is used for $\xRSd{n}$ and $\xKRSd{n}$.   
\end{lemma}

A further  consequence is that the above methods reach a vanishing 
residual at the break-down index $n_{bd}$. 
\begin{lemma} 
Let $x_{n}$ denotes any element of  $\{\xKSd{n}, \xKRSd{n}, \xRSd{n} \}$. Then we have 
\[ A x_{n_{bd}} -\yd = 0. \]
Conversely, $n < n_{bd}$ if and only if  $\|A x_{n_{bd}} -\yd\| >0$. 
\end{lemma}
\begin{proof}
This follows from Proposition~\ref{propbreak}.  Since at $n_{bd}$, there exists 
a degree-$n$ polynomial that interpolates at the $n-1$ eigenvalues and at $0$ 
(see also \cite[Theorem 3.5]{Hanke}), and only an interpolating polynomial 
can reach a zero residual. 
In particular we conclude that  $\|A x_{n_{bd}} -\yd\| >0$ if and only if  $n < n_{bd}$.
\end{proof}

The rational methods also have a relation to  
iterated Tikhonov  regularization: we use the following notation for the 
denominators in Proposition~\ref{pro_1}:
\begin{equation}\label{def:g} \g{n}(x) := \frac{1}{\prod_{i=1}^{n} (\frac{x}{\alpha_i} + 1) },  
\qquad \gk{n}(x) := \frac{1}{\prod_{i=1}^{\lfloor \frac{n}{2} \rfloor} (\frac{x}{\alpha_i} + 1) }.  
\end{equation}
These functions 
are  relevant since they  corresponds to the residual function of the 
iterated Tikhonov regularization (e.g., \cite{EHN}): This  is recursively defined  via 
\[ (\Op + \alpha_n I)  \xit{n} 
 = \alpha_n \xit{n-1} + \y, \qquad x_0 = 0, \]
and  the corresponding residual  can be expressed as 
\begin{equation}\label{defya}   
\begin{split} \yd  - A \xit{n} &= \g{n}(A A^*) \yd, \\
\ya{n} &:= \g{n}(A A^*) \yd  \quad (=  \yd  - A \xit{n} ).
\end{split}
 \end{equation}

For the next result, we set up a notation for 
the  polynomials appearing in the CGNE iteration: 
\begin{definition}\label{defqq}
Let $y\in Y$ be fixed, and 
 let $x_k$, for $k < n_{bd}$,  be the generated sequence of CGNE iterations. 
We denote the iteration polynomials (with parameters the iteration number $k$ and the right-hand side $y$)  of CGNE by  $\qcg{k-1}{y}$ 
and the   corresponding residual polynomials by $\pcg{k}{y}$:
\begin{alignat*}{2}  x_k &=: \qcg{k-1}{y}(A^*A)A^*y,& \qquad &\qcg{k-1}{y} \in \PS^{k-1}.  \\ 
 \pcg{k}{y}(\lambda) &:= 1 - \lambda \qcg{k-1}{y}(\lambda),&  
 &\pcg{k}{y} \in \PS_1^{k},  
\end{alignat*}  
i.e., 
\[ \yd - A x_k = \pcg{k}{y}(AA^*)y, \]
where $x_k$ is the CGNE iteration for the normal equation $A^*A x = A^*y$.
\end{definition}

The representation in Proposition~\ref{pro_1} suggest a factorization of the residual into an
CGNE-type one  and that of  an iterated Tikhonov method. 
This is established in the next proposition, which is also   
the basis for our convergence analysis. Note that $\ya{n}$ in \eqref{defya} is the residual function 
of the iterated Tikhonov method with $n$ steps, 
and we consider CGNE methods that use $\ya{n}$ as right-hand side, 
yielding the polynomials $\qcg{k}{\ya{n}}$ and $\pcg{k}{\ya{n}}$, where $k$ represents the iteration index of the CGNE method. 
\begin{proposition}\label{pro_3}
Let $n <  n_{bd}$. 
Then, with the notation of Definition~\ref{defqq} and \eqref{defya}, it holds that 
 \begin{align}\label{Aone}  \xRSd{n}  =   \xit{n} + \qcg{n}{\ya{n}}(A^*A)A^*[y - A  \xit{n}]. \end{align}
Moreover, the residual function can be represented as the following product:
\begin{align}\label{Atwo}
 \yd- A \xRSd{n} =\pcg{n}{\ya{n}}(AA^*) \g{n}(AA^*) \yd.  
 \end{align}
Analogously,  
 \begin{align}\label{Aonek}  \xKRSd{n}  =   \xit{k} + \qcg{n}{\ya{k}}(A^*A)A^*[y - A  \xit{k}]  \qquad
 k = \lfloor \frac{n}{2} \rfloor\, .
 \end{align}
and 
\begin{align}\label{Atwok}
 \yd- A \xKRSd{n} =\pcg{n}{\ya{k}}(AA^*) \g{k}(AA^*) \yd  \qquad  k = \lfloor \frac{n}{2} \rfloor\, .
 \end{align}
%
\end{proposition}
\begin{proof}
By Proposition~\ref{pro_1} we can write  
 $\xRSd{n} = p_{n-1}(A^*A)\g{n}(A^*A) A^*\yd$ with some $p_{n-1} \in \PS^{n-1}$.
Thus, we get $\yd - A\xRSd{n} = [I-   p_{n-1}(A^*A)\g{n}(A^*A) A^*A]\yd$. 
By definition of $\xRSd{n}$, the expression 
$\|\yd - A\xRSd{n}\| = \|\tilde{p}_n(AA^*)\g{n}(A A^*)  \yd\|$ 
is minimized over  polynomials $\tilde{p}_n \in \PS_1^n$. However, the minimizing polynomial so obtained 
is the residual of the CG-iteration when applied to the right-hand side $\g{n}(AA^*) \yd = \ya{n}$, 
i.e., in our notation $\pcg{k}{\ya{n}}$. Thus, we have the two representations 
\begin{align*} \yd - A\xRSd{n} &= \pcg{k}{\ya{n}}(A A^*)  \g{n}(A A^*) \yd, \qquad \text{ as well as  } \\
  \yd - A\xRSd{n}  & = [1-   p_{n-1}(A^*A)\g{n}(A^*A) A^*A]\yd. 
\end{align*}
The first equality  proves \eqref{Atwo}.
Equating the two right-hand sides, it follows that 
\[  p_{n-1}(x) = \frac{ 1- \pcg{k}{\ya{n}}(x)  \g{n}(x)}{ x \g{n}(x) } = 
 \frac{1}{x}\left( \frac{1}{\g{n}(x)} - \pcg{k}{\ya{n}}(x) \right).
\]
On the other hand, identity \eqref{Aone}  can be written as 
  $\xRSd{n} = h(AA^*) A^*\yd$, where   
\begin{align*}
 h(x) &= \frac{1- \g{n}(x) }{x} + \frac{1- \pcg{k}{\ya{n}}(x) }{x} \g{n}(x) 
 = 
  \frac{1}{x}\left(\frac{1}{\g{n}(x)}  -  \pcg{k}{\ya{n}}(x)  \right)  \g{n}(x).
\end{align*}
Comparing the two formula gives that  $h(x) = p_{n-1}(x) \g{n}(x)$ and thus  \eqref{Aone}.
The same proof goes through for $\xKRSd{n}$ by replacing  $\g{n}$ by $\g{k}$.
\end{proof}

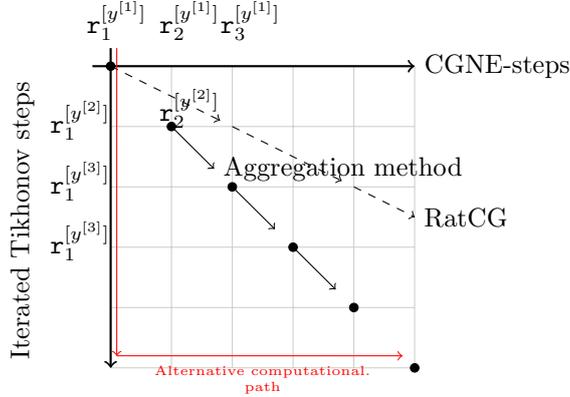
\begin{figure}
\begin{center}
\begin{tikzpicture}[scale=0.8]

  \draw[step=1cm,gray!40,very thin] (0,-5) grid (5,0);

  \draw[->,thick] (0,0.3) -- (0,-5) ;
  \draw[->,thick] (-0.3,0) -- (5,0) node[right]{CGNE-steps};

  \node[rotate=90, right=-10pt, left=8pt] at (-1.1,-0.1) {Iterated Tikhonov steps};

  \draw[->,] (1,-1) -- (1.7,-1.7) node[right]{Aggregation method};
  \draw[->,] (2,-2) -- (2.7,-2.7) ;
   \draw[->,] (3,-3) -- (3.7,-3.7) ;
   
      \draw[->,dashed] (0,-0) -- (1.8,-0.9) ;
   
         \draw[->,dashed] (2,-1) -- (3.8,-1.9);
         
            \draw[->,dashed] (4,-2) -- (5.,-2.5)          node[right]{RatCG};

     \draw[->,red] (0.1,0.3) -- (0.1,-4.8) ;     
      \draw[->,red] (0.1,-4.8) -- (4.8,-4.8) ;   
      \node at (2.5,-5.2) { {\tiny \color{red} \begin{tabular}{c} Alternative computational. \\ path \end{tabular}}};
         
   \node at (0.1,0.75) {$\pcg{1}{\ya{1}}$}; 
    \node at (1.3,0.75) {$\pcg{2}{\ya{1}}$}; 
      \node at (2.3,0.75) {$\pcg{3}{\ya{1}}$};
      \node at (-0.5,-0.9) {$\pcg{1}{\ya{2}}$}; 
     \node at (-0.5,-1.9) {$\pcg{1}{\ya{3}}$}; 
       \node at (-0.5,-2.9) {$\pcg{1}{\ya{3}}$};   
          \node at (1.3,-0.7) {$\pcg{2}{\ya{2}}$};   
      
  \foreach \x in {0,1,2,3,4,5}
    \filldraw[black] (\x,-\x) circle (2pt);

\end{tikzpicture}
\end{center}
\caption{Schematic description of the steps of the aggregation method.}\label{description}
\end{figure}

The previous result can  be utilized to design an 
alternative but equivalent  algorithm for computing $\xRSd{n}$ (or $\xKRSd{n}$): Calculate $n$ (respectively    
$k = \lfloor \frac{n}{2} \rfloor\,$) steps of an iterated Tikhonov regularization. For the 
resulting residual $\ya{n}$ as right-hand side, apply $n$ steps of a CGNE iteration and 
then add $\xit{n}$. 

A schematic description of the method is given in 
Figure~\ref{description}: On the $x$-axis we have CGNE steps, while the $y$-axis corresponds to 
``iterated Tikhonov steps''. The aggregation method is defined as the corresponding 
diagonal sequence. The mentioned alternative algorithm would compute a sequence 
by first going down and then steps to the left (indicated in red). 
The RatCG method could be represented 
analogously by a line with half the inclination as the aggregation method. 

\begin{notation}\label{not} For the analysis below,  having in mind 
Figure~\ref{description}, we now introduce some more concise notation. 
We use 
the notation that {\em upper indices in brackets} for iteration variables and residuals
refer to the $y$-axis, i.e., the  iterated Tikhonov steps, while {\em lower indices} refer to the CGNE-steps (on the $x$-axis).  

%
The functions with a circumflex $\hat{}$ refer to those of the RatCG method, while the ones 
without to the aggregation method:
\[ \pf{k}{n}:=  \pcg{k}{\ya{n}}(x) \qquad \pfk{k}{n}:=  \pcg{k}{\ya{\lfloor \frac{n}{2} \rfloor}}(x). \] 
i.e., these are the residual polynomials of the CGNE iteration with right-hand side $\ya{n}$ or 
$\ya{\lfloor \frac{n}{2} \rfloor}$
and the  index $k$ is the iteration index.
The residual function for the aggregation method  and 
the RatCG method 
are  then  
\[ \rR{n}(x) = \pf{n}{n}(x) \g{n}(x), \qquad 
 \rKR{n}(x) = \pfk{n}{n}(x) \gk{n}(x), 
\]
respectively, 
and Proposition~\ref{pro_3} then reads 
\[ y- A \xRSd{n} =  \rR{n}(AA^*)\yd, \qquad   y- A \xKRSd{n} =  \rKR{n}(AA^*)\yd. \]
\end{notation}

As a consequence of Proposition~\ref{pro_3}, we can  conclude that $ \xRSd{n}$ and $\xKRSd{n}$
a-priori do {\em not} represent  regularization methods since they are 
{\em not} continuous with respect to $\yd$. A situation that is  completely analogous to (and an implication of) 
the CGNE-case (see, e.g. \cite[Theorem~2.11]{Hanke}, \cite[Theorem~7.6]{EHN}). 
\begin{proposition}\label{disc}
Let $A^\dagger$ be unbounded. 
For fixed $n$, 
the mappings $\yd \to   \xRSd{n}$ and $\yd \to \xKRSd{n}$ are discontinuous. 
\end{proposition}
\begin{proof}
Theorem 2.11 in \cite{Hanke} identifies the set of discontinuities of the mapping $\yd \to \xKSd{n}$
as the set where $\yd$ belongs to an $n-1$-dimensional subspace of $AA*$. 
Consider two right-hand sides $\ydi$ and $\ydii$.
Let $x_1$ and $x_2$ denote the corresponding $\xRSd{n}$ with $\yd =  \ydi$ and  $\yd =\ydii$, 
and let $\yai{n}$ and  $\yaii{n}$ be the corresponding element in place of $\ya{n}$,
respectively. 
By 
\eqref{Aonek} we have 
\begin{equation}\label{hehe} x_1 - x_2 =   \xiti{n} - \xitii{n} + 
 \qcg{n}{\yai{n}}(A^*A)A^*\yai{n} -  \qcg{n}{\yaii{n}}(A^*A)A^*\yaii{n}.   
\end{equation} 
The first term $\xiti{n} - \xitii{n}$ tends to $0$ by classical regularization theory for 
 $\ydi \to \ydii$.  However, the second term does not:
 We may choose $\yai{n} \to \yaii{n}$, where $\yaii{n}$ is a set of discontinuity 
 of CGNE, such that  the second term in \eqref{hehe} does not vanish asymptotically:
  $x_1 - x_2 \not \to 0$.
 Let $\ydi = {(\g{n})}^{-1}(AA^*)\yai{n}  $ and $\ydii = {(\g{n})}(AA^*)^{-1}\yaii{n} $. 
Since the operator ${(\g{n})}(AA^*)^{-1}$ is continuous, we have $\ydi \to \ydii$, 
which proves the non-continuity. The proof for $\xKRSd{n}$ is completely analogous.
\end{proof}
Since  ${(\g{n})}(AA^*)^{-1}$ maps an $n-1$-dimensional  invariant subspace to itself, 
the proof reveals that the set of discontinuity points is identical to 
that for the CGNE method, namely the $n-1$-dimensional invariant subspaces of $AA^*$. 

In spite of that  and again analogous to the CGNE method (see \cite{Hanke}), 
the methods of interest in this paper become a regularization method if
combined with the discrepancy principle as a stopping rule. 
The factorization of the residual into those of two regularization methods, 
\eqref{Atwo}, \eqref{Atwok} 
paves the way for a convergence analysis: 
In fact, if both methods were linear, the analysis would be quite straightforward
(a comparable situation of a residual factorization arises for Nesterov acceleration~\cite{Ki2021}). 
However, a major difficulty arises because the CGNE method, and hence the rational methods as well, 
depend nonlinearly on the data 
(via the CGNE iteration index and the iterated Tikhonov index in $\ya{n}$). 
All this  requires to establish some deeper estimates that we present in the following sections.

\subsection{Collection of results for orthogonal polynomials}\label{sec:coll}
It is well-known that the residual functions of the CGNE iteration are represented 
by orthogonal polynomials on an interval. 
In the following, we collect classical results from the theory of 
orthogonal polynomials \cite{Szego} that we need later on.

To set the framework, in this section, we 
consider orthogonal polynomials on an interval $[0,T]$, which are 
orthogonal with respect to a nonzero positive measure $d\mu$, for which 
all moments $\int_0^T x^n d\mu(x)$ exists. 
As in \cite{Hanke} we define 
\[ \kappa:=  \text{number of nonzero points of increase
of } \mu(z) = \int_0^z d\mu(x). \] 
If there are (countably or uncountably many such points), we set $\kappa = \infty$. 
Then, the classical theory  ensures the existence 
of  polynomials $p_0,p_1 \ldots, p_{\kappa}$
of degrees $0,1,\ldots \kappa$, which are orthogonal 
with respect to the inner product $L^2([0,T],d\mu)$; see  
\cite{Szego} and \cite[p.~11]{Hanke}. 
We remark that the classical results for $\kappa = \infty$ are expounded in 
the monograph by Szeg\H{o} \cite{Szego}; there  the corresponding orthogonal polynomials 
constitute an orthogonal bases of $L^2([0,T],d\mu)$. 
In the 
degenerate case that $\kappa < \infty$, which is needed here, the sequence 
of $p_i$ constitutes only an orthogonal basis for a finite-dimensional subspace, 
and the corresponding results can be found in \cite{Hanke}.

\begin{lemma}\label{lem:lemma1}
Let $d\mu$ be a nonzero positive measure on $[0,T]$ as before. 
Let $p_k(x) \in \P^k$, $k \in \N_0$, $k = deg(p_k)$, $k \leq  \kappa$   
be a (possible finite) sequence of orthogonal polynomials with respect to $L^2([0,T],d\mu)$, 
normalized with $p_k(0) = 1$.
Then all zeros of $p_k$ are real, simple,  and in the interior of $[0,T]$. 
Let 
$\lambda_{j,k}$ be the  $j$th (ordered) zero of the polynomial $p_k$, $1 \leq j \leq k$.
Then the zeros of $p_k$  and  $p_{k+1}$ interlace: 
\begin{equation}\label{interlace} 0 <  \lambda_{1,k+1} < \lambda_{1,k} < \lambda_{1,k+1} <  \lambda_{2,k} <  \ldots < 
\lambda_{k,k} <
\lambda_{k+1,k+1} < T. \end{equation}
\end{lemma} 
See \cite[Theorem 3.3.1 and Theorem 3.3.2]{Szego},
for $\kappa = \infty$ and \cite[p.~11]{Hanke} for $\kappa < \infty$. 

As a consequence, by the normalization, we may express $p_k$ as 
\begin{equation}\label{rep} p_k(x) = \prod_{j=1}^k (1 -\tfrac{x}{\lambda_{j,k}} ).  \end{equation}
Of particular importance is the first (=smallest) zero $\lambda_{1,k}$ of $p_k$. 
Let us explicate some well-known statements. We denote by $p_k'$ the derivative of $p_k$.

\begin{lemma}\label{lem:lemma2} 
Let $p_k$ be normalized orthogonal polynomial in the same setup as in Lemma~\ref{lem:lemma1}, 
with $k \leq  \kappa$ and $\lambda_{j,k}$ the ordered zeros as above. 

Then the following results hold: 
\begin{enumerate}
\item\label{en:one} $p_k(x)$ is convex for $x  \in [0,\lambda_{1,k}]$.
\item We have 
\begin{alignat}{2} 
  p_k'(0) &= -\sum_l \lambda_{l,k}^{-1}  &  \label{en:i} \\ 
     \lambda_{1,k} &\geq |p_k'(0)|^{-1} &  \label{en:ia} \\
 |p_k(x)| &\leq 1&   &x \in [0,\lambda_{1,k}]  \label{en:ii} \\
0 \leq \frac{1-p_k(x)}{x} &\leq |p_k'(0)|&  &x \in [0,\lambda_{1,k}]  \label{eniii} \\
      p_k^2(x) \frac{\lambda_{1,k}}{\lambda_{1,k}-x} &\leq 1& &x \in [0,\lambda_{1,k}]  \label{one} \\
   p_k^2(x) \frac{\lambda_{1,k}}{\lambda_{1,k}-x} x^\nu &\leq \cnu{\nu}  p_k'(0)^{-\mu}&\quad   &x \in [0,\lambda_{1,k}],     
   \ \nu >0 \
   \cnu{\nu} = \nu^\nu. \label{two} 
\end{alignat}
 \end{enumerate}
\end{lemma}
The statement in item~\ref{en:one} is a consequence of the interlacing property~\eqref{interlace};
see \cite[(2.19)]{Hanke}. The estimate \eqref{en:i}, \eqref{en:ia}, and \eqref{en:ii} 
follows from \eqref{rep}  \cite[(2.19)]{Hanke}, 
while \eqref{eniii} follows from the convexity (see \cite[p.~188]{EHN}).
The inequality \eqref{one} can be found in \cite[p.~42]{Hanke}, \cite[p.~185]{EHN},
and \eqref{two} is \cite[(3.10)]{Hanke}.

\subsubsection{Residual polynomials for the CGNE method}
It is well-known that the residual polynomials for the CGNE method are 
orthogonal polynomials as in the previous section. 
To state the measure, we let $F_\lambda$ be the spectral family of
$AA^* :N(A^*)^\bot \to N(A^*)^\bot $ and $E_\lambda$ that of $A^*A: N(A)^\bot \to N(A)^\bot$. 
In the case of a  compact 
operators $A$, we may write them in terms of the singular system $(\sigma_n,u_n,v_n)$
of $A$ as  
\[ E_\lambda(x) = \sum_{\sigma_i^2 < \lambda} (x,v_i)_X v_i 
 \qquad 
 F_\lambda(x) = \sum_{\sigma_i^2 < \lambda} (x,u_i)_Y u_i . 
\]
The relation between $\pcg{k}{y}$ and orthogonal polynomials is as follows: 
\begin{proposition}\label{p4}
Given $y \in Y$. 
Define the measure on $[0, \|A\|^2]$ 
by\linebreak  \mbox{$d \mu(\lambda) := \lambda d \|F_\lambda y\|^2$}, 
and let $k\leq \kappa$. 
Then the measure $\mu$ satisfies the condition in Section~\ref{sec:coll},
and the residual functions $\pcg{k}{y}$ are the normalized orthogonal polynomials 
associated to $\mu$. In particular we have that 
$\pcg{k}{y}$ are orthogonal with respect to the inner product 
in $L^2([0, \|A\|^2], d\mu)$.  
\end{proposition}
This is \cite[Corollary 7.4]{EHN}. In particular, the 
estimates of Lemma~\ref{lem:lemma1} and \ref{lem:lemma2} can be 
employed for $\pcg{k}{y}$ in place of $p_k$ as long as $k \leq \kappa$. 
In this context, we have 
\[ \kappa = \text{maximal dimension of the Krylov spaces  $\KS^n$ generated by $\yd$} = 
n_{bd} -1. 
\]

\subsubsection{Iterated Tikhonov method}
We also recall some estimates for the iterated Tikhonov method (with variable regularization 
parameters $\alpha_i$). 
The standard reference here is the article by Hanke and Groetsch \cite{HankeGroetsch}. 
Recall that the residual function for the iterated Tikhonov method
is given by $\g{n}$ in \eqref{def:g}.

Define, for a sequence of nonnegative $\alpha_k$ the quantity 
 \[ \sigma_n := \sum_{k=1}^n \alpha_k^{-1}. \]
 
\begin{lemma}[\mbox{\cite[Lemma 2.3]{HankeGroetsch}}]\label{lemmafour}
For the sequence $\alpha_n$, assume that there exists a constant $\ca$ such that for all $n$ 
\begin{equation}\label{eqeq} \frac{1}{\alpha_n} \leq \ca \sigma_{n-1}. \end{equation}
Then there exist a a $\cnuit{\nu}$ such that  
 \begin{equation}\label{twentyfor}
  |\g{n}(\lambda) \lambda^\nu | \leq \cnuit{\nu} \sigma_n^{-\nu} \qquad 0 \leq  \nu < n. 
 \end{equation}
In case that $0 < \nu \leq 1$, the result \eqref{twentyfor} holds without the  condition \eqref{eqeq}.  
\end{lemma}

Following the notation of Lemma~\ref{lem:lemma1} and the Notation~\ref{not}, 
we denote the roots of the rational residual polynomials
(recall that a circumflex represents the RatCG method) as follows: 
\begin{align} \label{defroots} 
\begin{split} \rootX{k}{n}{j} &=  \text{ $j$th (ordered) zero of the polynomial $ \pf{k}{n} $} , 1 \leq j \leq k, 
k \leq n. \\ 
\rootXk{k}{n}{j} &=  \text{ $j$th (ordered) zero of the polynomial $ \pfk{k}{n} $} , 1 \leq j \leq k, 
k \leq n. 
\end{split} 
\end{align}
We define the inner product 
\begin{equation}\label{defscp} \scp{\phi,\psi}:= \int_0^{\|T\|^2+} \psi(\lambda) \phi(\lambda)  \|F_\lambda \yd\|^2 . \end{equation}
For brevity, when using this notation, 
we do not explicitly write the argument $\lambda$ in the functions in the inner product, 
i.e., $\rR{n}= \rR{n}(\lambda)$, $\g{n} = \g{n}(\lambda)$, $p_{n-1}= p_{n-1}(\lambda)$, etc.  

Define the measures   
\begin{equation}\label{lemma11} d\betan{n} := \lambda (\g{n})^2 dF_\lambda \|\yd\|^2 \qquad d\betank{n}:= \lambda (\gk{n})^2 dF_\lambda \|\yd\|^2. \end{equation}
Proposition~\ref{p4} then states that 
the polynomials $\pf{k}{n}$ (respectively $\pfk{k}{n}$) are orthonormal with respect to 
$ d\betan{n}$ (respectively $ d\betank{n}$). In particular,   
\begin{equation}\label{mainmain} 
\begin{split} \skp{\pf{k}{n}, \lambda {(\g{n})}^2 p_{n-1}} = 0, \\
 \skp{\pfk{k}{n}, \lambda {(\gk{n})}^2 p_{n-1}} = 0, 
\end{split}
\qquad \text{ for all  }p_{k-1} \in  \P^{n-1}.  
\end{equation}
Moreover, the optimality condition holds by definition 
\begin{equation} \label{opti} 
\begin{aligned} \|A \xRSd{n} -\yd\| &\leq \|p_n(A A^*) \g{n}(AA^*)\yd\| \qquad \\
 \|A \xKRSd{n} -\yd\| &\leq \|p_n(A A^*) \gk{n}(AA^*)\yd\| 
  \end{aligned} \text{ for all }
 p_{n} \in  \PS_1^n.
\end{equation}
We note that much of the analysis below is equal for the aggregation method and 
for the RatCG method, which is why we mostly only prove results for the first method. 
In fact the RatCG method can be subsumed under the aggregation method by 
formally setting  the odd parameters to $\alpha_{2n+1} = \infty$.

\section{Convergence results} 
We now investigate the convergence rates properties of 
the aggregation method, given by $\xRSd{n}$ and of the RatCG method, 
$\xKRSd{n}$. 
As it is common, for establishing convergence rates, we 
impose an abstract smoothness condition in form of a source condition \cite{EHN}:
There exists a $\mu>0$ and an $\omega \in X$ such that 
\begin{equation}\label{sc}
\xd = (A^*A)^\mu \omega.   
\end{equation}
For notational purpose, we define 
\[ \mu_*:= (\mu + \tfrac{1}{2}). \] 
We recall definition \eqref{defroots} such that 
$\rooti{k}{n}= \text{ the smallest zero of $\pf{k}{n}$}$  and\linebreak  
$\rootik{k}{n}= \text{ the smallest zero of $\pfk{k}{n}$.}$
Furthermore, let  
\[ \phio{k}{n}(\lambda):= 
  \pf{k}{n}(\lambda) \g{n}(\lambda) \left(\frac{\rooti{k}{n}}{\rooti{k}{n}-\lambda} \right)^\frac{1}{2}.
\]
Orthogonality implies the following estimate; see \cite[(7.7)]{EHN} or \cite[Proof of Lemma~3.7]{Hanke}:
\begin{equation}\label{phio} 
\begin{split}
 \|y^\delta-A \xRSd{n}\|^2 & =
 \int_0^{\|A\|^2+} |\pf{n}{n}(\lambda)|^2  \g{n}(\lambda)^2  d F_\lambda  \|\yd\|^2    \\
& 
 \leq  \int_{0}^{\lambda_{n,1}} |\pf{n}{n}(\lambda)|^2 \g{n}(\lambda)^2 \frac{\rooti{n}{n}}{\rooti{n}{n}-\lambda} d E \|\yd\|^2 \\
 & = \|F_{\rooti{n}{n}}\phio{n}{n}(AA^*) y_\delta\|^2. 
\end{split} 
\end{equation}
We do not explicitly state the analogous identity for $y^\delta-A \xKRSd{n}$, where  obvious modifications have to be made. 
The next result is analogous to \cite[Lemma 3.7]{Hanke}.
\begin{lemma}
Let $x^\dagger$ satisfy a source condition \eqref{sc},  and assume that $1 \leq n < n_{bd}$ holds. 
We have   the estimate (with $\cgen$ from \eqref{two})
\begin{equation}\label{dpest} \|A \xRSd{n}-  y^\delta\|  \leq \delta  + \cgen |{\pfprime{n}{n}}(0)|^{-\mu_*}  \| w\|. \end{equation} 
Similarly, 
\begin{equation}\label{dpestcg}  \|A \xKRSd{n}-  y^\delta\| \leq \delta  + \cgen |{\pfkprime{n}{n}}(0)|^{-\mu_*}  \| w\|. \end{equation}
\end{lemma}
\begin{proof} 
Using \eqref{phio}, $y = A \xd$, and \eqref{sc},  we get 
\begin{align*} 
   \|A x_{\RS,n}\ - y^\delta\| & \leq 
   \|F_{\rooti{n}{n}}\phio{n}{n}(AA^*) (y_\delta- y)\| 
    + 
      \|F_{\rooti{n}{n}}\phio{n}{n}(AA^*) A (A^*A)^\mu w \| .
      \end{align*} 
By  $\g{n} \leq 1$ and 
  \eqref{one} with $\pf{n}{n}$ in place of $p_k$, we obtain  
  \[  \|F_{\rooti{n}{n}}\phio{n}{n}(AA^*) (y_\delta- y)\|  \leq \delta. \] 
With  \eqref{two} for   $p_k = \pf{n}{n}$  and $\g{n} \leq 1$, we find 
for the second term 
\begin{align*} 
 &\|F_{\rooti{n}{n}}\phio{n}{n}(AA^*) A (A^*A)^\mu w \|^2 \\
&= \int_0^{\rooti{n}{n}}   \frac{\rooti{n}{n}}{\rooti{n}{n}-\lambda} 
  {(\pf{n}{n})}^2\lambda^{2 \mu_*} (\g{n})^2 d F_\lambda \|w\|^2 \\
&\qquad \leq \cgen^2
|{\pfprime{n}{n}}(0)|^{-2 \mu_*} 
\int_0^{\rooti{n}{n}} d F_\lambda \|w\|^2 
 \leq \cgen^2
|{\pfprime{n}{n}}(0)|^{-2 \mu_*}  \|w\|^2.
 \end{align*}
The proof for $\xKRSd{n}$ is completely analogous as only orthogonality and 
 $\gk{n} \leq 1$ is needed. 
\end{proof}

In the next step we estimate the error 
$\xRSd{n} -\xd$ (cf.~the proof of Lemma 3.8 in \cite{Hanke} or \cite[Lemma 7.11]{EHN}). 
\begin{lemma}
Let $x^\dagger$ satisfy a source condition \eqref{sc},  and assume  $1 \leq n < n_{bd}$. 

Let $0 <\epsilon \leq \rooti{n}{n}$   be arbitrary. 
Then, 
 \begin{align}\label{preest}
  \| \xRSd{n}  -x^\dagger\| \leq 
  \epsilon^{-1/2} \|A \xRSd{n} - y\| + 
  (\sigma_n +  |\pfprime{n}{n}(0)|)^{\frac{1}{2}}  
  \delta  + 
  \epsilon^\mu \|w\|.
 \end{align}
The similar estimate holds for  $\xKRSd{n}$: 
For  $0 <\epsilon \leq \rootik{n}{n}$
and $\sigk_n: = \sigma_{\lfloor \frac{n}{2} \rfloor}$,  
 \begin{align}\label{preestk}
  \| \xKRSd{n}  -x^\dagger\| \leq 
  \epsilon^{-1/2} \|A \xRSd{n} - y\| + 
  (\sigk_n +  |\pfkprime{n}{n}(0)|)^{\frac{1}{2}}  
  \delta  + 
  \epsilon^\mu \|w\|.
 \end{align}
\end{lemma}

\begin{proof}
We only proof the statement for $\xRSd{n}$. The case of $\xKRSd{n}$ only requires obvious modifications. 
Let $\epsilon \leq \rooti{n}{n}$.  In the proof of Proposition~\ref{pro_3}, we established the representation 
\[ \xRSd{n} =  p_{n-1}(A^*A)\g{n}(A^*A) A^*\yd \quad \text{ with } \ 
 p_{n-1}(x) = x^{-1}(\g{n}(x)^{-1} - \pf{k}{n}(x)). \]
Note that $p_{n-1}$ is a polynomial that depends on $\yd$. 
We define an artificial iterate corresponding to the exact data:
\[ x_{n} :=  p_{n-1}(A^*A)\g{n}(A^*A) A^*y =:h(A^*A)A^*y 
%
\]
with 
\[ 
 h(x) = p_{n-1}(x) \g{n} =  \frac{1}{x}(1 - \g{n}) +  \frac{1}{x}(1 -  \pf{k}{n}) \g{n}.
\]
By a using the spectral projector $E_\epsilon$ with respect to $A^* A$ we have 
 \begin{align*}   \xRSd{n}  - x^\dagger &= 
 (I - E_\epsilon) ( x_{n}^\delta  - x^\dagger) + 
E_\epsilon  (x_{n}^\delta  - x^\dagger) \\
&  =  (I - E_\epsilon) (x_{n}^\delta    - x^\dagger)  + E_\epsilon ( x_{n}^\delta  -  x_n) + 
E_\epsilon ( x_n- x^\dagger). 
\end{align*}
We estimate the norms of the three terms separately: 
\begin{align*}
  \|(I - E_\epsilon) &(x_{n}^\delta    - x^\dagger) \|^2 \\
  &=
  \int_{\epsilon}^{\|A\|^2+} d E_\lambda \|x_{n}^\delta    - x^\dagger \|^2 \leq 
  \epsilon^{-1} \int_{\epsilon}^{\|A\|^2+} \lambda d E_\lambda \|x_{n}^\delta    - x^\dagger \|^2 \\
  &\leq  \epsilon^{-1} \int_{0}^{\|A\|^2+} \lambda d E_\lambda \|x_{n}^\delta    - x^\dagger \|^2 
 =    \epsilon^{-1} \|A x_{n}^\delta -  y\|^2.   
%
 \end{align*} 
 The second term, using the commutativity of spectral projectors $E_\epsilon A^*A = AA^* F_\epsilon$, 
 can be estimated as 
 \begin{align*}
\| E_\epsilon (  x_{n}^\delta  -  x_n) \|^2  &= 
\int_0^\epsilon h(\lambda)^2 d E_\epsilon \|A^*(\yd -y)\|^2 = 
\int_0^\epsilon h(\lambda)^2 \lambda d F_\epsilon \|\yd -y\|^2  \\
&  \leq  \|(\yd -y)\|^2 \left(\sup_{\lambda \leq \epsilon} |h(\lambda) \lambda| \right)
\left(\sup_{\lambda \leq \epsilon} |h(\lambda) | \right) .
  \end{align*} 
%
We have that  $0 \leq \g{n} \leq 1$ and  by \eqref{en:ii}
that $0\leq \pf{k}{n}(\lambda)) \leq 1$ for $\lambda \leq \epsilon$  (recalling that $\epsilon \leq \rooti{n}{n}$).
We may write 
\begin{equation} \label{help11}\lambda h(\lambda) = 1 - \g{n}(\lambda) \pf{k}{n}(\lambda),  
\end{equation}
and thus $0 \leq \lambda h(\lambda)  \leq 1$ for  $\lambda \leq \epsilon$.  
Noting the $\g{n} \leq 1$ we have for $\lambda \leq \epsilon$, 
\[ |h(\lambda)| \leq 
|\frac{1}{\lambda}(1 - \g{n}(\lambda))| +  
|\frac{1}{\lambda}(1 -  \pf{n}{n}(\lambda))| \leq 
\sigma_n +
|\pfprime{n}{n}(0)|. 
\]
Here we  used the result \cite[(15)]{HankeGroetsch} for the first estimate and  \eqref{eniii} for  the second. 

Finally the third terms is bounded as
\begin{align*}
\|E_\epsilon ( x_n- x^\dagger)\| = 
\|E_\epsilon \left(I - A^*A h(A^*A)\right) (A^*A)^\mu w\| \leq 
 \|w\| \sup_{\lambda \leq \epsilon} |(1-\lambda h(\lambda))\lambda^\mu|.  
\end{align*}
Using \eqref{help11} and 
 noting that 
$\pf{k}{n}(\lambda))\g{n}(\lambda) \in [0,1]$ for $\lambda \leq \epsilon$
yields the last estimate of the form $\|w\| \epsilon^\mu$.
%
\end{proof}

We arrive at the first main error bound:
\begin{proposition}
Define the residual 
\[ \rho_n:= \|A \xRSd{n}  -\yd\|. \] 
Assume a source condition \eqref{sc} to hold, and let $1 \leq n<  n_{bc} $.
Then 
 \begin{align}\label{secondest}
  \| \xRSd{n} -x^\dagger\| &\leq  \rho_n^{\frac{\mu}{\mu+ 1/2}}
   \|w\|^{\frac{1/2}{\mu+ 1/2}}
  + (\sigma_n + |\pfprime{n}{n}(0)|)^{1/2}  \delta +
   |\pfprime{n}{n}(0)|^{1/2} \rho_n 
  \end{align}
and with 
\[ \hat{\rho}_n:= \|A \xKRSd{n}  -\yd\|. \] 
 \begin{align}\label{secondestk}
  \| \xKRSd{n} -x^\dagger\| &\leq  \hat{\rho}_n^{\frac{\mu}{\mu+ 1/2}}
   \|w\|^{\frac{1/2}{\mu+ 1/2}}
  + (\sigk_n + |\pfkprime{n}{n}(0)|)^{1/2}  \delta +
   |\pfkprime{n}{n}(0)|^{1/2} \hat{\rho}_n .
  \end{align}
\end{proposition}
\begin{proof}
Again we only display the proof for $\xRSd{n}$.
Starting from \eqref{preest}, we may improve the estimate  
by balancing the two $\epsilon$-terms: 
 Define 
\[ \epsilon_*^{\mu + 1/2} :=  \frac{\rho_n}{\|w\|}. \]
In case that $\epsilon_* < \rooti{n}{n}$, it can be substituted into 
\eqref{preest}, and \eqref{secondest} is valid even  without the last term
on the right-hand side. 

 Otherwise, in case 
$ \epsilon^* >  \rooti{n}{n}$,
we take $\epsilon = |\pfprime{n}{n}(0)|^{-1}$ in \eqref{preest}, 
which is by \eqref{en:ia}  less or equal than $\rooti{n}{n}$,
and thus, $ |\pfprime{n}{n}(0)|^{-1} < \epsilon^*$. Using this estimate 
after plugging in our choice of $\epsilon$ in  \eqref{preest} yields 
 \begin{align*}
  \| x_k^\delta -x^\dagger\| &\leq 
  |\pfprime{n}{n}(0)|^{1/2} \rho_k +   |\pfprime{n}{n}(0)|^{-\mu} \|w\|+  
  (\sigma_n + |\pfprime{n}{n}(0)|)^{1/2}  \delta\\
  &\leq 
     |\pfprime{n}{n}(0)|^{1/2} \rho_k +    \epsilon_*^\mu  \|w\|+  (\sigma_n + |\pfprime{n}{n}(0)|)^{1/2}  \delta \\
    & \leq   \rho_k^{\frac{\mu}{\mu+ 1/2}}  \|w\|^{\frac{1/2}{\mu+ 1/2}}  + 
      (\sigma_n + |\pfprime{n}{n}(0)|)^{1/2}  \delta + |\pfprime{n}{n}(0)|^{1/2} \rho_k .     
 \end{align*}
\end{proof}

Having the discrepancy principle in mind, where $\rho_k \sim \delta$, 
we observe that the first term in \eqref{secondest} is already of optimal order,
and it remains to find the appropriate bounds 
for $|\pfprime{n}{n}(0)|$.  This represents the main analytical work in this paper 
since we cannot fully rely on classical estimates as in \cite{Hanke} because 
 $\pf{n}{n}$ is a double-indexed ``diagonal sequence'' of orthogonal 
polynomials, which raises several difficulties.

\section{Estimates for the discrepancy principle.}
We are now considering a stopping rule in form of 
a discrepancy principle. 
That is, 
we terminate the respective method at an index $n_*$ such that,  for a selected $\tau > 1$, 
both of the following conditions hold:
\begin{align} \|A \xRSd{n_*-1} -\yd\| & \geq  \tau \delta 
\label{dpone} \\ 
 \|A \xRSd{n_*} -\yd\| & <  \tau \delta
 \label{dptwo} 
\end{align}  
The analogous definitions with $\xKRSd{n_*-1}$, $\xKRSd{n_*}$  are omitted for brevity. 

To allow for the case  $n_* = 1$, 
 we extend the definition of  $\xRSd{n}$ by setting 
\begin{equation}  
\begin{split} &\xRSd{0}:= 0, \quad  \xKRSd{0}:= 0,
\qquad \text{ thus, } \\
& A \xRSd{0} -\yd = -\yd
\quad \text{ and }  \quad 
A \xKRSd{0} -\yd = -\yd.  
\end{split}
\end{equation} 
The monotonicity of the residual in Lemma~\ref{lemma:mono},
implies that  $\|A \xRSd{n} -\yd\|  \geq  \tau \delta$  
 also  holds for all $n \leq n_*-1$. 

To rule out trivial cases, we impose the following condition on the data:  
\begin{equation}\label{firstcond}  
\text{For some $\tau_2 > \tau >1$ it holds that } 
\|\yd\| \geq \tau_2 \delta . 
\end{equation}
The condition \eqref{firstcond} requires that there is ``more signal than noise'' in 
the data,  and it is only violated in pathological cases. 
 
\begin{lemma} 
Let $\delta >0$ and assume \eqref{firstcond}. 
Then there exists a finite $ n_* \geq 1$ such that \eqref{dpone}--\eqref{dptwo} holds. 
\end{lemma}
\begin{proof} 
  By definition of \eqref{firstcond}, we have 
that the discrepancy condition\linebreak  \mbox{$\|A \xRSd{n} -\yd\|  \geq  \tau \delta$}    holds for $n= 0$. 
By Lemma~\ref{lemma:mono}, the residuals 
 $\|A \xRSd{n} -\yd\|$ and  $\|A \xKRSd{n} -\yd\|$  are  bounded by $\|A \xKSd{n} -\yd\|$. 
It is well-known \cite[p.~46]{Hanke} that for the CGNE method 
an $n$ exists such that $\|A \xKSd{n} -\yd\| <  \tau \delta$. Hence 
there also must exist an $n_*$ with \eqref{dpone}--\eqref{dptwo}.
\end{proof}

\begin{lemma}\label{lemmasix}
Let $\delta >0$, and let \eqref{firstcond} hold.
Assume  a source condition  \eqref{sc}, and  let \eqref{dpone} hold. 
Then $1 \leq n_* < n_{bd}$ and 
\begin{align*}  |{\pfprime{m}{m}}(0)|^{-\mu_*} 
&>  \frac{\tau-1}{\cgen \|w\|} \delta \quad \text{ for } m \leq n_*-1  \\ 
    |{\pfkprime{m}{m}}(0)|^{-\mu_*}  &>  \frac{\tau-1}{\cgen \|w\|} \delta \quad \text{ for } m \leq n_*-1 .
    \end{align*}
\end{lemma}
\begin{proof}
 Proposition~\ref{propbreak} and \eqref{dpone} implies that $n_* < n_{bd}$. 
 Then \eqref{dpest} can be used, and the left-hand side can be replaced by $\tau \delta$.  
 This yields the bound for $|{\pfprime{m}{m}}(0)|^{-\mu_*}$
 and $|{\pfkprime{m}{m}}(0)|^{-\mu_*}$
 for any $m \leq n_*-1$.
\end{proof}


Following \cite[p.~189]{EHN}, we now 
define the degree-($n-1$) polynomials 
\begin{align}
 \we{n-1}(x) &:= \frac{1}{x}\left(  \pf{n-1}{n-1}(x) -\pf{n}{n}(x)  \right),  \label{def:we}\\
  \wek{n-1}(x) &:= \frac{1}{x}\left(  \pfk{n-1}{n-1}(x) -\pfk{n}{n}(x)  \right),  \label{def:wek}
\end{align} 
with their values at $x= 0$, 
\begin{align}
  \pin_{n-1} &:=  \pfprime{n-1}{n-1}(0) -\pfprime{n}{n}(0)  = \we{n-1}(0), \label{def:wetwo} \\
  \pink_{n-1} &:=  \pfkprime{n-1}{n-1}(0) -\pfkprime{n}{n}(0)  = \wek{n-1}(0). \label{def:wektwo}
\end{align} 

From results below on the monotonicity of the zeros (see \eqref{propx}), it follows  that  $\pin_{n-1} > 0 $
and  $\pink_{n-1} > 0 $.
As in \cite[Section 7.3]{EHN} \cite[Section 3.3]{Hanke}, we try to estimate $\pin_{n-1}$, $\pink_{n-1}$ by using results 
about the smallest zero of  $\we{n-1}$, $\wek{n-1}$. 

\begin{lemma}\label{lemmaneun}
Let $\epsilon >0$, $\delta >0$, $2 \leq n < n_{bd}$,
and assume that a source condition \eqref{sc} holds. 
%
%
Then 
\begin{align}\label{eq:lemmaneun} 
\begin{split}
\|A \xRSd{n-1} - \yd\|  &\leq 
\sup_{x \in [0,\epsilon]} \left|\frac{\we{n-1}(x)}{\pin_{n-1}} \g{n-1}(x)\right|
\left(\delta + \epsilon^{\mu + 1/2} \|w\|  \right)  \\
& \qquad +  
\frac{1}{\pin_{n-1}}\| (I-F_\epsilon) \we{n-1}(AA^*) \g{n-1}(AA^*)  \yd\| ,
\end{split}
\end{align}
as well as 
\begin{align}\label{eq:lemmaneunk} 
\begin{split}
\|A \xKRSd{n-1} - \yd\|  &\leq 
\sup_{x \in [0,\epsilon]} \left|\frac{\wek{n-1}(x)}{\pink_{n-1}} \gk{n-1}(x)\right|
\left(\delta + \epsilon^{\mu + 1/2} \|w\|  \right)  \\
& \qquad +  
\frac{1}{\pink_{n-1}}\| (I-F_\epsilon) \wek{n-1}(AA^*) \gk{n-1}(AA^*)  \yd\| .
\end{split}
\end{align}
\end{lemma}
\begin{proof}
We only prove the case of  $\xRSd{n-1}$; the proof for $\xKRSd{n-1}$ is identical 
subject to a change of notation. 

We have that $\frac{1}{\pin_{n}}\we{n} \in \PS_1^n$. 
By  the minimality of the residual in \eqref{opti}, using $p = \frac{1}{\pin_{n-1}}\we{n-1}$, we find 
\begin{align*}  \|A \xRSd{n-1} - \yd\|  &\leq \frac{1}{\pin_{n-1}} \|\we{n-1}(AA^*) \g{n-1}(AA^*)  \yd\| 
 \\
 &\leq \frac{1}{\pin_{n-1}} \| F_\epsilon \we{n-1}(AA^*) \g{n-1}(AA^*)  \yd\| \\
 & \qquad + 
 \frac{1}{\pin_{n-1}} \| (I-F_\epsilon) \we{n-1}(AA^*) \g{n-1}(AA^*)  \yd\| \\
 & \leq 
 \frac{1}{\pin_{n-1}} \| F_\epsilon \we{n-1}(AA^*) \g{n-1}(AA^*)  (\yd-y)\|  \\
& \qquad  + \frac{1}{\pin_{n-1}} \| F_\epsilon \we{n-1}(AA^*) \g{n-1}(AA^*)  y\|  \\
& \qquad  +  \frac{1}{\pin_{n-1}} \| (I-F_\epsilon) \we{n-1}(AA^*) \g{n-1}(AA^*)  \yd\|. 
\end{align*}
Let $W_\epsilon:= \sup_{x \in [0,\epsilon]} |\frac{\we{n-1}(x)}{\pin_{n-1}} \g{n-1}(x)|$ Then
since $\g{n-1} \leq 1$, 
\[  \frac{1}{\pin_{n-1}} \| F_\epsilon \we{n-1}(AA^*) \g{n-1}(AA^*)  (\yd-y)\|   \leq 
W_\epsilon  \|F_\epsilon   (\yd-y)\| \leq W_\epsilon   \delta. 
\]
In a similar manner, 
\begin{align*}\frac{1}{\pin_{n-1}} &\| F_\epsilon \we{n-1}(AA^*) \g{n-1}(AA^*)  y\| \\
&= 
\frac{1}{\pin_{n-1}} \| F_\epsilon \we{n-1}(AA^*) \g{n-1}(AA^*)  A (A^*A)^\mu w\|  \\
&\leq  W_\epsilon \| F_\epsilon  A (A^*A)^\mu w\|  \leq W_\epsilon 
\left(\int_0^{\epsilon} \lambda^{1 + 2 \mu} d E_\lambda \|w\|^2 \right)^\frac{1}{2} \leq 
W_\epsilon  \|w\| \epsilon^{\mu + \frac{1}{2}}. 
\end{align*}
\end{proof}
The last term in \eqref{eq:lemmaneun} requires some special treatment for which we 
present some technical lemmata.
\begin{lemma}\label{kabel}
Let $2 \leq n <  n_{bd}$. Then, 
\begin{align}
 \| A \we{n-1}(AA^*) &\g{n-1}(AA^*)  \yd\|^2 \label{lastterm0} \\
 & = \pin_{n-1}\left(
  \|A \xRSd{n-1}-\yd\|^2 -
 \|A \xRSd{n}-\yd\|^2 \right) \label{lastterm1}  \\
 & \qquad \qquad + 
  \alpha_n^{-2} \| A^*(A \xRSd{n}-\yd)\|^2. \label{lastterm2}
  \end{align}
\end{lemma}
\begin{proof}
In our  notation (cf.~\ref{defscp}),
we have   \[\| A \we{n-1}(AA^*) \g{n-1}(AA^*)  \yd\|^2
=\scp{ \we{n-1} \g{n-1}, \lambda \we{n-1} \g{n-1}},\] 
and by definition $\lambda \we{n-1}(\lambda) =  \pf{n-1}{n-1}(\lambda) -\pf{n}{n}(\lambda) $ 
and $\we{n-1} = \pin_{n-1} + \lambda p_{n-2}$, where $p_{n-2} \in \P^{n-2}$. 
Thus, 
\begin{align*}
 & \scp{ \we{n-1} \g{n-1}, \lambda \we{n-1}\g{n-1}}
 = \scp{(\pin_{n-1} + \lambda p_{n-2}) {(\g{n-1})}^2,  \pf{n-1}{n-1}-\pf{n}{n}} \\
 & \qquad =  \scp{\pin_{n-1}  {(\g{n-1})}^2,  \pf{n-1}{n-1}(\lambda) } 
 -\scp{(\pin_{n-1} + \lambda p_{n-2})  {(\g{n-1})}^2,  \pf{n}{n}(\lambda)},
\end{align*}
where we used the orthogonality of $\pf{n-1}{n-1}$ to  $p_{n-2}$.

Since $  \pf{n-1}{n-1}(\lambda) = 1 + \lambda q_{n-2}$ with $q_{n-2} \in \P^{n-2}$
we may use again the orthogonality \eqref{mainmain}  of $\pf{n-1}{n-1}$ 
to replace $\pin_{n-1}$ in the first term by $\pin_{n-1} \pf{n-1}{n-1}$:
\begin{align*} &\scp{\pin_{n-1} {(\g{n-1})}^2,  \pf{n-1}{n-1} }
 = \pin_{n-1} \scp{\pf{n-1}{n-1}  {(\g{n-1})}^2,   \pf{n-1}{n-1} } \\
 & \qquad =  \pin_{n-1} \|A \xRSd{n-1}-\yd\|^2.  
\end{align*}
Regarding the second term: 
We have $\g{n-1} = (\frac{\lambda}{\alpha_n} +1) \g{n}$. Then, 
\begin{align}
& \scp{(\pin_{n-1} + \lambda p_{n-2})  {(\g{n-1})}^2,  \pf{n}{n}}= 
 \scp{(\pin_{n-1} + \lambda p_{n-2}) (\tfrac{\lambda}{\alpha_n} +1)^2  {(\g{n})}^2,  \pf{n}{n}} 
 \nonumber \\
 & \qquad = 
 \pin_{n-1} \scp{{(\g{n})}^2,  \pf{n}{n}}  + 
 \scp{ \lambda p_{n-2} (\tfrac{\lambda}{\alpha_n} +1)^2  {(\g{n})}^2,  \pf{n}{n}}.
 \label{tagit}
\end{align} 
Here we used the orthogonality of $\pf{n}{n}$ in \eqref{mainmain} 
and $n \geq 2$ to replace 
$\pin_{n-1} (\tfrac{\lambda}{\alpha_n} +1)^2$ by  $\pin_{n-1}$.  
Since $\pf{n}{n}(\lambda) = 1 + \lambda q_{n-1}$ with some $q_{n-1} \in \P^{n-1}$, with 
the same orthogonality argument we get for the first term
\[  \pin_{n-1} \scp{{(\g{n})}^2,  \pf{n}{n}} = 
 \pin_{n-1} \scp{\pf{n}{n} {(\g{n})}^2,  \pf{n}{n}} =
 \pin_{n-1} \|A \xRSd{n}-\yd\|^2.  \] 
The last term in  \eqref{tagit} is handled as follows
(using that $\lambda \pin_{n-1} \in \P^{n-1}$):
\begin{align*}  p_{n-2}(\lambda)& (\tfrac{\lambda}{\alpha_n} +1)^2
 = \tfrac{\lambda^2}{\alpha_n^2}  p_{n-2}  + \P^{n-1} 
 = 
 \tfrac{\lambda}{\alpha_n^2} ( \pin_{n-1} + \lambda p_{n-2})  + \P^{n-1}  \\ 
 &=  
  \tfrac{1}{\alpha_n^2} \lambda \we{n-1}  + \P^{n-1} = 
 \tfrac{1}{\alpha_n^2}\pf{n-1}{n-1}(\lambda) -\tfrac{1}{\alpha_n^2}\pf{n}{n}(\lambda) + \P^{n-1} \\
 & = 
 -\tfrac{1}{\alpha_n^2}\pf{n}{n}(\lambda) +  \P^{n-1}. 
\end{align*} 
By orthogonality 
we are allowed to replace 
$p_{n-2} (\frac{\lambda}{\alpha_n} +1)^2$ by $-\frac{1}{\alpha_n^2} \pf{n}{n}(\lambda)$ in 
the last term in  \eqref{tagit}: 
\begin{align*}
& \scp{ \lambda p_{n-2}) (\tfrac{\lambda}{\alpha_n} +1)^2  {(\g{n})}^2,  \pf{n}{n}} =
 -\alpha_n^{-2} \scp{ \lambda  \pf{n}{n} {(\g{n})}^2,  \pf{n}{n}} \\
 & \qquad =  -\alpha_n^{-2} \| A^*(A \xRSd{n}-\yd)\|^2. 
\end{align*}
Collecting all results proves the statement. 
\end{proof}

For the case of $\xKRSd{n}$, we have an analogous statement:  
\begin{lemma}
Let $2 \leq n <  n_{bd}$. Then, 
\begin{align}
 \| A &\wek{n-1}(AA^*) \gk{n-1}(AA^*)  \yd\|^2 \label{lastterm0k} \\
 & \quad = \pink_{n-1}\left(
  \|A \xKRSd{n-1}-\yd\|^2 -
 \|A \xKRSd{n}-\yd\|^2 \right) \\
 & \qquad \qquad + 
 \| A^*(A \xKRSd{n}-\yd)\|^2 \times \begin{cases} 
 \alpha_k^{-2} & n = 2k, \\ 
 0  & n = 2k+1.
 \end{cases} 
  \label{lastterm2k}
  \end{align}
\end{lemma}
\begin{proof}
Again the proof is identical to that of Lemma~\ref{kabel}. 
The only difference is how the terms $\gk{n-1}$ and $\gk{n}$ are 
related, namely by either $\gk{n-1} = \gk{n}$ in case $n$ is odd, or 
$\gk{n-1} = (\frac{\lambda}{\alpha_k} +1) \gk{n}$ in case of 
$n$ even. Note that the 
odd case corresponds to a difference to two CGNE-residual polynomials with respect to the same measure; 
in this case, the identity in the lemma is \cite[Equation~(7.16)]{EHN}. 
\end{proof}

\begin{lemma}
Let $2 \leq n <  n_{bd}$ and 
 $\alpha_n \geq \epsilon$.  Then 
\begin{align*}\| (I-F_\epsilon)& \we{n-1}(AA^*) \g{n-1}(AA^*)  \yd\|^2 \\
 &
 \leq 
 \frac{4}{3 \epsilon} 
\left(
  \|A \xRSd{n-1}-\yd\|^2 -
 \|A \xRSd{n}-\yd\|^2 \right) .
\end{align*} 
Similarly, under the same conditions, 
 \begin{align*} \| (I-F_\epsilon)& \wek{n-1}(AA^*) \gk{n-1}(AA^*)  \yd\|^2 \\
  & 
  \leq 
 \frac{4}{3 \epsilon} 
\left(
  \|A \xKRSd{n-1}-\yd\|^2 -
 \|A \xKRSd{n}-\yd\|^2 \right) .
  \end{align*}
\end{lemma}
\begin{proof}
We first focus on the term in \eqref{lastterm2}. 
Up to the factor $\alpha_n^{-2}$ it is equal to 
$\scp{ \lambda  \pf{n}{n} {(\g{n})}^2,  \pf{n}{n}} $. 
By orthogonality, we may add the degree-($n-1$) polynomial $\pf{n-1}{n-1}(\lambda)$ at one argument of 
the inner product, yielding 
\begin{align*} 
 &\scp{ \lambda  \pf{n}{n} {(\g{n})}^2,  \pf{n}{n}} 
 = \scp{ \lambda  \left(\pf{n}{n} - \pf{n-1}{n-1} \right)\g{n},  \g{n} \pf{n}{n}} \\
 & \leq  
 \scp{\lambda  \left(\pf{n}{n} - \pf{n-1}{n-1}\right){\g{n}}, 
 \left(\pf{n}{n} - \pf{n-1}{n-1} \right){\g{n}} }^\frac{1}{2}  \scp{ \lambda  \pf{n}{n} {\g{n}},  \pf{n}{n}{\g{n}}} ^\frac{1}{2}, 
\end{align*}
where we used the Cauchy-Schwarz inequality. Thus, 
\begin{align} \label{yace}
\text{\eqref{lastterm2}} = \scp{ \lambda  \pf{n}{n} {(\g{n})}^2,  \pf{n}{n}}   \leq 
  \scp{ \lambda  \left(\pf{n}{n} - \pf{n-1}{n-1} \right) {\g{n}},  \left(\pf{n}{n} - \pf{n-1}{n-1} \right) \g{n} }.
\end{align}
Furthermore, 
\[  \left(\pf{n}{n} - \pf{n-1}{n-1} \right){\g{n}} = 
 -\frac{1}{\frac{\lambda}{\alpha_n} + 1 } \lambda   \we{n-1}(\lambda) \g{n-1}.
\]
The term in  \eqref{lastterm0} (left-hand side of the identity) 
is $\scp{\lambda   \we{n-1}\g{n-1} ,  \we{n-1} \g{n-1}}$. 
Using  \eqref{yace}, 
the estimate  \eqref{lastterm1}--\eqref{lastterm2}  implies 
\begin{align}\label{doschowieda}
\begin{split} 
 \scp{\lambda   \we{n-1} \g{n-1} ,  \we{n-1} \g{n-1}}  &\leq \pin_{n-1}\left(
  \|A \xRSd{n-1}-\yd\|^2 -
 \|A \xRSd{n}-\yd\|^2 \right)  \\
 & \qquad + 
 \scp{ \lambda \left(\frac{ \frac{\lambda}{\alpha_n}}{\frac{\lambda}{\alpha_n}+1}\right)^2  \we{n-1} \g{n-1}, 
 \we{n-1} \g{n-1}}.
 \end{split}
\end{align}
Define the function 
\[ \ha(z):= 1 - \left(\frac{z}{1 + z}\right)^2= \frac{1 + 2 z}{(1 +z)^2}. \]
Then by moving the last term in \eqref{doschowieda} to the right-hand side, we get 
\begin{align} \label{sstwo}
 \begin{split}
&\scp{\lambda  \ha( \frac{\lambda}{\alpha_n})  \we{n-1} \g{n-1},  \we{n-1} \g{n-1}} \\
& \qquad \qquad \leq  
\pin_{n-1}\left(
  \|A \xRSd{n-1}-\yd\|^2 -
 \|A \xRSd{n}-\yd\|^2 \right). 
 \end{split} \end{align}
Since $\ha(z)$ is monotonically decreasing for positive arguments, 
we have  with $\alpha_n \geq \epsilon$, $ \ha( \frac{\lambda}{\alpha_n})  \geq  \ha( \frac{\lambda}{\epsilon})$. 
The function
$\lambda \ha( \frac{\lambda}{\epsilon})$ is monotonically increasing for positive 
$\lambda$ and thus for $\lambda \geq \epsilon$, 
\[ \lambda  \ha( \frac{\lambda}{\epsilon}) \geq 
 \epsilon  \ha( \frac{\epsilon}{\epsilon})  = \frac{3}{4} \epsilon. 
\]
This gives the final estimate: 
\begin{align*}
 \| (I-F_\epsilon)& \we{n-1}(AA^*) \g{n-1}(AA^*)  \yd\|^2  \\
 &= \int_{\epsilon}^{\|A\|^2+}
 \we{n-1}(\lambda)^2 \g{n-1}(\lambda)^2  dF_\lambda(\yd)\\
 &\leq 
 \frac{4}{3 \epsilon} 
 \int_{\epsilon}^{\|A\|^2+} 
 \lambda  \ha( \frac{\lambda}{\epsilon}) 
 \we{n-1}(\lambda)^2 \g{n-1}(\lambda)^2  dF_\lambda(\yd)  \\
 & \leq 
  \frac{4}{3 \epsilon} 
 \int_{0}^{\|A\|^2+} 
 \lambda  \ha( \frac{\lambda}{\alpha_n}) 
 \we{n-1}(\lambda)^2 \g{n-1}(\lambda)^2  dF_\lambda(\yd) \\
 &\leq_{\eqref{sstwo}}   \frac{4}{3 \epsilon}  \pin_{n-1}\left(
  \|A \xRSd{n-1}-\yd\|^2 -
 \|A \xRSd{n}-\yd\|^2 \right) \\
 &\leq 
 \frac{4}{3 \epsilon}  \pin_{n-1}   \|A \xRSd{n-1}-\yd\|^2 .
\end{align*}
The case of $\xKRSd{n-1}$ requires only minor modifications. 
For the even case $n=2k$, the proof is identical
with the same $g$ as in the preceding steps, while for 
the odd case, we may formally use $\alpha_n = \infty$ such that 
the last term in \eqref{doschowieda} vanishes. In this case, we may set 
$g(z) = 1$ and proceed as before. (The constant $\frac{3}{4}$ could be set to $1$ then.)
\end{proof} 
In combination with Lemma~\ref{lemmaneun}, under all the required assumptions, we arrive at the following proposition:

\begin{proposition}
 Let $\epsilon >0$, $\delta >0$, $2 \leq n < n_{bd}$, $ \alpha_n \geq \epsilon$, and let 
 a source condition \eqref{sc} be satisfied. 
 Then, with $C = \frac{2}{\sqrt{3}}$, the estimate holds:  
 \begin{align}\label{thnn1}
 \begin{split}
\|A \xRSd{n-1} - \yd\|  &\leq  
\sup_{x \in [0,\epsilon]} \left|\frac{\we{n-1}(x)}{\pin_{n-1}} \g{n-1}(x)\right|
\left( \delta + \|w\| \epsilon^{\mu_*}\right)  \\
&\qquad \quad + 
 \frac{C}{\pin_{n-1}^\frac{1}{2}\epsilon^\frac{1}{2}}  \|A \xRSd{n-1} - \yd\|.
\end{split} 
\end{align}
Similarly, 
 \begin{align}\label{thnn1k}
 \begin{split}
\|A \xKRSd{n-1} - \yd\|  &\leq  
\sup_{x \in [0,\epsilon]} \left|\frac{\wek{n-1}(x)}{\pink_{n-1}} \gk{n-1}(x)\right|
\left( \delta + \|w\| \epsilon^{\mu_*}\right)  \\
&\qquad \quad + 
 \frac{C}{\pink_{n-1}^\frac{1}{2}\epsilon^\frac{1}{2}}  \|A \xKRSd{n-1} - \yd\|.
\end{split} 
\end{align}
\end{proposition}

%

\subsection{Roots of polynomials} 
We now discuss results for the roots of the residual polynomials 
and extends some known results to our needs. 
We assume throughout 
\[ 2 \leq n < n_{bd}. \] 
Recall the definition of the roots in \eqref{defroots}. 
The results of Lemma~\ref{lem:lemma1} state that 
these are real-valued inside the interval $[0,\|A\|^2]$ and the 
interlacing property  holds:  
\[ \rootX{k}{n}{i} < \rootX{k-1}{n}{i} < \rootX{k}{n}{i+1} \qquad i < k-1, k \leq n, \]
and  analogously for $\rootXk{k}{n}{i}$. 
Furthermore, we define  
\begin{equation}\label{def:muone} 
\begin{split} 
\muone{n-1}:= \text{smallest root of 
 $ \we{n-1}(\lambda)= \frac{1}{\lambda}  \left(\pf{n-1}{n-1} -  \pf{n}{n} \right)$,} \\ 
 \muonek{n-1}:= \text{smallest root of 
 $ \wek{n-1}(\lambda)= \frac{1}{\lambda}  \left(\pfk{n-1}{n-1} -  \pfk{n}{n} \right)$} \\ 
 \end{split} 
\end{equation}
Lemma~\ref{lemma18} in the appendix verifies that the roots of $\we{n-1}$ and $\wek{n-1}$ are all 
real so that this definition makes sense. 
The next three lemmata, Lemma~\ref{interlace0}--\ref{interlace2},
are also proven in the appendix. 
First of all, we have the monotonicity in the upper index $n$:
\begin{lemma}\label{interlace0}
Let $2 \leq n < n_{bd}$. 
We have 
\[  \rootX{k}{n}{i} < \rootX{k}{n-1}{i}  \qquad  \text{ as well as }  \qquad 
\rootXk{k}{n}{i} \leq  \rootXk{k}{n-1}{i} . \]
\end{lemma} 
Together with the interlacing result, we obtain 
$\rootX{n-1}{n-1}{i} > \rootX{n-1}{n}{i} > \rootX{n}{n}{i}$
and 
$\rootXk{n-1}{n-1}{i} \geq  \rootXk{n-1}{n}{i} > \rootX{n}{n}{i}$. 
An immediate consequence is that both $\pin_{n-1}$ and $\pink_{n-1}$ in \eqref{def:wetwo} are  positive:  
For instance, 
\begin{align} 
\begin{split}\label{propx} \pin_{n-1}&= \pfprime{n-1}{n-1}(0) -\pfprime{n}{n}(0)   
= -\sum_{i=1}^{n-1}  \frac{1}{\rootX{n-1}{n-1}{i}} + 
\sum_{i=1}^{n} \frac{1}{\rootX{n}{n}{i}}\\
&> 
 -\sum_{i=1}^{n-1} \frac{1}{\rootX{n-1}{n-1}{i}}  + \sum_{i=1}^{n-1} \frac{1}{\rootX{n-1}{n-1}{i}}
 + \frac{1}{\rootX{n}{n}{n}} > 0.
 \end{split}
 \end{align}

We have to relate $\muone{n-1}$ to the smallest roots of 
the residual polynomials:
\begin{lemma}\label{interlace1}
Let $2 \leq n < n_{bd}$. Then,
\begin{equation}\label{lalala}
\rootX{n-1}{n}{1} \leq  \muone{n-1}, \end{equation}
and 
\begin{equation}\label{anoy} \min\left\{  \rootX{n-1}{n-1}{1}, \rootX{n}{n}{2} \right\}  \leq  \muone{n-1}. \end{equation}
The analogous  result holds for the roots $\rootXkpure{}{}{}$ and $\muonek{n-1}$  in place of  
$\rootXpure{}{}{}$ and $\muone{n-1}$. 
\end{lemma}

\begin{lemma}\label{interlace2}
Let $2 \leq n < n_{bd}$. 
The following estimate holds: 
\begin{equation} \rootX{n-1}{n-1}{1} \leq 
 e^{2} \max\left\{\frac{\|A\|^2}{\alpha_n^2},1\right\}  
 \rootX{n-1}{n}{1} .\end{equation}
The analogous  result holds for the roots $\rootXkpure{}{}{}$ in place of  
$\rootXpure{}{}{}$, where $\alpha_n$ is replaced by $\alpha_k$ for 
$n = 2k$ and by $\infty$ in case of $n = 2k+1$.
\end{lemma}
 
\subsection{The main convergence result}
We continue with the convergence analysis.  
Using the previous lemmata, we obtain the central estimate for  $\pfprime{n}{n}(0)$
and ${\pfkprime{n}{n}}(0)$:
\begin{lemma}\label{sixteen}
Let $\delta \in (0,\delta_0]$ with a fixed $\delta_0$, let \eqref{firstcond} hold, and  
assume that the discrepancy principle is satisfied 
\eqref{dpone} and \eqref{dptwo}   with $n_*\geq 2$. 
Assume a parameter choice with a fixed constant $c_0$  such that for all $n$,
\begin{equation}\label{pc} \alpha_n  \geq c_0.  \end{equation}
Then there is a constant $D$ such that 
\begin{equation}\label{sixseven} 
\begin{split} |{\pfprime{n}{n}}(0)| &\leq  D \delta^{-1/(\mus)}, \qquad \text{ as well as } \\ 
  |{\pfkprime{n}{n}}(0)| &\leq D \delta^{-1/(\mus)} .
\end{split}
 \end{equation}
\end{lemma}

\begin{proof}
\newcommand{\eS}{\kappa}
Again, we only proof the result for $\pfprime{n}{n}(0)$, while that for $\pfkprime{n}{n}(0)$ is 
identical. Set 
\[ \epsilon = \xi  \delta^{1/\mus},\qquad \xi >0, \]
where $\xi$ is a parameter that satisfies the inequalities 
\[ \xi \leq \frac{c_0}{\delta_0^{1/\mu_*}}, \quad 
1 +\|w\| \xi^{\mu^*} < \tau, \quad 
\xi \leq \eS e^{-2} \min\{1, \frac{c_0}{\|A\|^2}, \}, 
\] 
and where we set $ \eS = \frac{\tau-1}{\cgen \|w\|} $.
Since $\tau >1$, we may conclude that such a $\xi$ exists. 
By the assumptions $\delta \leq \delta_0$ and \eqref{pc}, 
this choice of $\xi$ leads to $\epsilon \leq \alpha_n$. 
By the discrepancy principle \eqref{dptwo}
and  Lemma~\ref{lemmasix} we obtain  $n < n_{bd}$ and  
$ \eS \delta^{1/\mus}  < |{\pfprime{n-1}{n-1}}(0)|^{-1}$, 
which is by \eqref{en:ia} smaller than 
  $\rootX{n-1}{n-1}{1}$. Thus, 
\begin{align*}  \epsilon &= \xi \delta^{1/\mu^*} 
\leq   \frac{\xi}{\eS} \rootX{n-1}{n-1}{1} \leq 
  \frac{\xi}{\eS} e^2 \max\left\{\frac{\|A\|^2}{\alpha^2},1\right\}   
   \rooti{n-1}{n} \\
   & \leq 
         \frac{\xi}{\eS} e^2 \frac{1}{ \min\{\frac{c_0^2}{\|A\|^2},1\}  }
          \rooti{n-1}{n}
           \leq_{\eqref{lalala}}
    \frac{\xi}{\eS} e^2 \frac{1}{ \min\{\frac{c_0^2}{\|A\|^2},1\}}  
          \muone{n-1}.
\end{align*}            
 By our choice of $\xi$,  it follows that 
 $\epsilon  <  \muone{n-1}$.
From Lemma~\ref{lemma18}, we conclude that $\we{n-1}$ has $n-1$ real roots, 
denoted as $\muone{i,n-1}$, 
and hence can be represented as 
$\we{n-1}(\lambda) = \pin_{n-1} \Pi_{i=1}^{n-1} (1 - \frac{\lambda}{\muone{i,n-1}})$.  
Similar as in \eqref{en:ii}, we may conclude that 
$\we{n-1}(x)$ is monotonically decreasing to its first root, and hence 
we find $\frac{\we{n-1}}{\pin_{n-1}} \leq 1$ for $\lambda  \leq \muone{n-1}$.
Together with $\g{n-1}\leq 1$ we obtain
$\sup_{x \in [0,\epsilon]} \left|\frac{\we{n-1}(x)}{\pin_{n-1}} \g{n-1}(x)\right| \leq 1$ and 
from \eqref{thnn1} 
\begin{equation}\label{intermed}
\|A \xRSd{n-1} - \yd\|
 \left( 1 -  \frac{C}{\pin_{n-1}^\frac{1}{2}\epsilon^\frac{1}{2}}  \right)
 \leq \delta ( 1+ \|w\| \xi^{\mu_*} \kappa^{\mu_*} ) 
 \end{equation}
We make  a case distinction: 

(i) In case that  $1 -\frac{C}{\pin_{n-1}^\frac{1}{2}\epsilon^\frac{1}{2}} \leq \frac{1}{2}$ 
%
%
we obtain 
\[ \pin_{n-1} < \frac{C^2}{4 \epsilon} = \frac{C^2}{4 \xi  \eS} \delta^{-1/(\mus)}  . \]

(ii) In the  opposite case to (i), the left-hand side in \eqref{intermed} is positive 
(and larger than $\frac{1}{2}$ and  
 we have  by the discrepancy principle 
\begin{align*} 
\tau \delta   (1 - \frac{C}{  \sqrt{\pin_{n-1}\epsilon} })  &\leq 
( 1+ \|w\| \xi^{\mu_*}  )  \delta. 
\end{align*}
We find 
 \[ \frac{1}{2} \leq ( 1-  \frac{C}{  \sqrt{\pin_{n-1}\epsilon} })  
\leq \tau^{-1}( 1+ \|w\| \xi^{\mu_*}  ) .
\]
Our  choice of $\xi$ ensures that  the right-hand side is strictly smaller 
than $1$.  It follows that $\frac{C}{  \sqrt{\pin_{n-1}\epsilon}}$ is bounded from below 
by some constant $\tilde{C}$ such that 
\[ \pin_{n-1} < \left(\frac{C}{\tilde{C}}\right)^2 \frac{1}{\epsilon} = \frac{C^2}{\tilde{C}^2 } \frac{1}{\xi}   \delta^{-1/(\mus)}  . \]
In either case this leads to $\pin_{n-1} \leq C \delta^{-1/(\mus)}$. 
%
Finally, using this and again Lemma~\ref{lemmasix}
\[ |{\pfprime{n}{n}}(0)| \leq |{\pfprime{n-1}{n-1}}(0)| + \pin_{n-1} \leq  D \delta^{-1/(\mus)} .\]
\end{proof}

We can now establish the main convergence theorem. 
\begin{theorem}\label{theorem1}
Let $\delta \in [0,\delta_0]$ with a fixed $\delta_0$, let \eqref{firstcond} hold, and  
assume that $n_*$ satisfies the discrepancy principle 
\eqref{dpone} and \eqref{dptwo}  with some   $\tau$,  $\tau_2 > \tau >1$. 
Let the regularization parameters satisfy 
\begin{equation}\label{alphacon} \alpha_i \geq c_0, \end{equation}
and assume the sequence $\alpha_n$ satisfies \eqref{eqeq}. 

Then we have the optimal-order convergence rates estimate  
  \begin{align*}
  \| \xRSd{n_*} -x^\dagger\| 
   & \leq C\delta^{\frac{\mu}{\mu+ 1/2}} 
  \end{align*}
  and 
    \begin{align*}
  \| \xKRSd{n_*} -x^\dagger\| 
   & \leq C\delta^{\frac{\mu}{\mu+ 1/2}} .
  \end{align*}
\end{theorem}
\begin{proof}
We first treat the case $n_*\geq 2$. 
In the following, we use $C$ as generic constant. 
For brevity we set $n:= n_*$ in the proof. 
By \eqref{dptwo}, we can use  $\rho_n  \leq \tau \delta $ in 
the error bound \eqref{secondest}. Together with 
\eqref{sixseven}, we find 
\begin{equation}\label{syz}
  \| \xRSd{n}  -x^\dagger\| \leq  C \delta^{\frac{\mu}{\mu+ 1/2}}
  + (\sigma_n + c \delta^{\frac{1}{\mu_*}} )^{1/2}  \delta.
  \end{equation}
By monotonicity of the residual, we also know that 
\[ \tau \delta \leq \| A  \xRSd{n-1}  -\yd\|  \leq 
\| A \xit{n-1} -\yd\|  
  \]
With \eqref{eqeq} and the estimates for the iterated Tikhonov 
from Lemma~\ref{lemmafour}, we have 
\begin{align*}  \tau \delta &\leq  \| A \xit{n-1} -\yd\| = \|\g{n-1}(AA^*) \yd\| \\
& \leq 
  \|\g{n-1}(AA^*) (\yd-y)\| +  \|\g{n-1}(AA^*) A A^*A w \| \leq 
  \delta + \cnuit{\mu_*}  \|w\|\sigma_{n-1}^{-\mu_*}.  
\end{align*}
Thus, it follows that $\sigma_{n-1} \leq C \delta^{1/\mu_*}$. 
Since $\sigma_n = \sigma_{n-1} + \frac{1}{\alpha_n}$ 
we find from \eqref{syz} the estimate 
 \begin{equation}\label{exact}
  \| \xRSd{n} -x^\dagger\| \leq  C \delta^{\frac{\mu}{\mu+ 1/2}}
  + \frac{1}{\sqrt{\alpha}}  \delta \leq 
  C \delta^{\frac{\mu}{\mu+ 1/2}} + \frac{1}{\sqrt{c_0}} \delta \leq 
   C \delta^{\frac{\mu}{\mu+ 1/2}}.
  \end{equation}
The case for $\xKRSd{n}$ is obtain by the same steps, with 
replacing $\sigma_n$ by $\sigk_{n}$ and 
$\xit{n-1}$ by $\xit{\lfloor n-1 \rfloor}$.

The case $n_*=1$ for $\xRSd{n}$ can be handled by explicit calculations. 
Note that this corresponds to a least-squares fit over 
the one-dimensional space $\xad{\alpha_1}$. The solution is 
$ \xKRSd{1} = c \xad{\alpha_1}$, with $c = \frac{\scp{A\xad{\alpha_1},\yd}}{\|A \xad{\alpha_1}\|^2}.$
A short calculation leads to a residual 
$\yd - A  \xKRSd{1} = \pf{1}{n}(AA^*) \g{1}\yd$ with $\pf{1}{n}= 1 - \gamma x$ and 
\[ \gamma = \frac{1}{\alpha}  \frac{\scp{A\xad{\alpha_1},\yd-A\xad{\alpha_1}}}{\|A \xad{\alpha_1}\|^2}. \]
Thus $|\pfprime{1}{n}(0)| = \gamma$. In view of \eqref{secondest}, we have to bound $\gamma$
by an expression $c \delta^{-1/\mu_*}$. 
We have by the Cauchy-Schwarz inequality and the triangle inequality 
\begin{align*} \gamma &\leq \frac{1}{\alpha} \frac{\| \yd-A\xad{\alpha_1}\|}{\|\yd\| - \| \yd-A\xad{\alpha_1}\|}
 \leq \frac{1}{\alpha} \frac{\tau \delta }{\|\yd\| - \tau \delta} \\
  &\leq \frac{1}{\alpha} \frac{\tau \delta }{\tau_2 \delta  - \tau \delta} \leq 
  \frac{1}{\tfrac{\tau_2}{\tau_1}- 1} \frac{1}{\alpha}  \leq C_2\delta^{-1/\mu_*}.
\end{align*}
Here, we used \eqref{firstcond} and the parameter choice 
\eqref{alphacon}, which implies $\alpha \geq c  \delta^{-1/\mu_*}$.
Now proceeding as in the first part of the proof 
following the derivation of \eqref{exact} settles this case. 

The case $n_*=1$ for $\xKRSd{n}$ is identical to that of a CGNE iteration, where in the 
first step, the iteration is terminated by the discrepancy principle. 
This case is well-known \cite{Hanke}.
\end{proof}
 
\begin{remark}\label{remrem}
 The mindful reader will notice that the condition  $\alpha_i \geq c_0$ 
is probably suboptimal.
(This can be observed from the term $\frac{1}{\sqrt{\alpha}}  \delta$ 
in \eqref{exact}, which exhibits the ``wrong'' rate in $\delta$.) 
 In fact, the ``expected'' required condition is 
that $\alpha_i \geq C \delta^{1/\mu_*}$. This lower bound is 
the optimal-order choice of the regularization parameter 
for a non-saturating method (like the iterated Tikhonov method with 
$n$ large enough.) \cite{EHN}. The stated condition \eqref{alphacon}
is only needed to have $\epsilon \leq \muone{n-1}$.  
If, by chance, $\muone{n-1} \geq c \delta^{1/\mu_*}$, then 
the choice $\alpha_i \geq C \delta^{1/\mu_*}$ yields the correct 
rate in \eqref{exact} and the optimal result. 
\end{remark}
We formulate the previous remark as a proposition:

\begin{proposition}
The result of  Theorem~\ref{theorem1} holds if 
\eqref{alphacon} is replaced by 
\begin{equation}\label{modchoice} \alpha_i \geq C \delta^{1/\mu_*}\end{equation}
and if additionally the condition $\muone{n-1} \geq c \delta^{1/\mu_*}$ holds true. 
The latter  condition is satisfied if one of the following 
statement hold: 
\begin{enumerate}
\item\label{itone}
$\pf{n-1}{n-1}$ interlaces with $\pf{n}{n}$. 
\item\label{ittow}  
The discrepancy condition $\|A {\xRSd{n-1}}^{[n]} -\yd\| \geq \tau \delta$ 
hold where ${\xRSd{n-1}}^{[n]}$ correspond to the iterate after 
$n-1$ step of the CGNE iteration \eqref{Aone}
at level $n$. (I.e., one CGNE step before $\xRSd{n}$). 
\item\label{itthree}
$\rootX{n-1}{n-1}{1} \leq C  \rootX{n-1}{n}{1}$ .
\end{enumerate}
\end{proposition}
\begin{proof}
The sketch of the convergence result with the modified regularization parameter 
condition is stated in Remark~\ref{remrem} and omitted. Regarding the 
conditions in the proposition, we note that if  item \ref{itone} holds, 
then $\rootX{n-1}{n-1}{1} < \lambda_{2,n}^n$. Using Lemma~\ref{interlace1} yields 
that $\rootX{n-1}{n-1}{1} < \muone{n-1}$ and thus  
in the proof of Lemma~\ref{sixteen} it follows that $\muone{n-1} \geq  c \delta^{1/\mu_*}$. 
In case that item~\ref{ittow} holds, we may use the estimate \eqref{dpest} to conclude 
$|{\pcg{n-1}{n}}'(0)|^{-\mu_*}  \geq c \delta$, which yields that  
$\rootX{n-1}{n}{1} \geq C \delta^{1/\mu_*}$. With \eqref{lalala} we again get 
$\muone{n-1} \geq  c \delta^{1/\mu_*}$. Finally, item \ref{itthree} 
can be used in a similar manner since $\rootX{n-1}{n-1}{1} \geq  c \delta^{1/\mu_*}$
by the discrepancy principle. 
\end{proof}
Unfortunately, we could not prove any of these conditions from properties of the 
polynomials and the discrepancy principle alone,  so that the previous  proposition 
remains a conditional proof of the suggested parameter choice. Numerically, 
we observed that \eqref{modchoice} works in practice.

The condition \eqref{eqeq} on the $\alpha_i$ is included only for generality 
and it is always satisfied in the setting of the theorem if at least one $\alpha_i$, $i < n_*$ 
is finite. We included the statement in the theorem since initially we did not explicitly rule 
out the case that all $\alpha_i$ except for $\alpha_{n_*}$ are $\infty$.

In any case,  the more essential condition \eqref{alphacon} (or \eqref{modchoice}) on an appropriate  
lower bound for the regularization parameter means that 
the $\alpha_i$ should be in the ``over-regularization'' region, i.e., larger 
(or equal) to the rate-optimal $\alpha$. This can be understood as the condition that 
we have to avoid oscillatory or ``noisy''  under-regularized solutions spoiling the 
rational Krylov space. We think that this is an important finding when generalizing 
the method to nonlinear setups.

For the rational methods, a  discussion about the required number of iterations  
has little benefit, since this highly depends on the specific selection 
of the regularization parameter. In fact, if the first regularization parameter 
is already selected as the optimal-order choice $\alpha = \delta^{1/\mu_*}$
(and the source condition is in the non-saturation region), then the first 
iteration is already acceptable, and the methods terminate at the first iteration. 
An upper bound for the number of iteration follows from Lemma~\ref{lemma:mono}
because it implies that the number of iterations for the rational methods is 
always smaller or equal than  that of the CGNE iteration (with the  discrepancy principle).

\section{Conclusion}
In our main result, Theorem~\ref{theorem1}, we show that the aggregation method and the 
RatCG method are regularization methods when combined with the discrepancy principle and 
when the regularization parameters satisfy a uniform lower bound. 
Our proofs are based on viewing the methods of interest as rational Krylov spaces, 
and the results try to emulate the corresponding analysis for the CGNE case. 
We furthermore conjecture that 
the lower bound for the $\alpha_i$  could possibly be replaced by forcing the $\alpha_i$ to be larger than 
the optimal-order regularization parameter choice (the ``over-regularization region'')
for the iterated Tikhonov method.

\appendix
\section{Results for the roots of the derived polynomials}\label{appendix}
In this appendix, we prove the Lemmata~\ref{interlace0}--\ref{interlace2}. 
At first, we need a result about monotonicity of 
the zeros of orthogonal polynomials 
when the measure satisfies a monotonicity condition.
We first discuss this for the roots of the polynomials $\pf{k}{n}$.
Define an auxiliary function (with a slight abuse of notation where the upper index is now the interpolation parameter)
that interpolates between $\g{n}$ and $\g{n-1}$:
\[ \gtau{\tau}:= \g{n}(1+  \tau \tfrac{\lambda}{\alpha_n}) 
= \frac{(1 +\tau \tfrac{\lambda}{\alpha_n})}{\Pi_{i=1}^n (1 +\tfrac{\lambda}{\alpha_i})},
\qquad \tau \in [0,1].\]
Define the associated measure (recall the definition of $\beta$ in \eqref{lemma11})
\[ d \betatau{\tau} :=  (1 + \tau \tfrac{x}{\alpha_n})^2 d \betan{n}  
= (1 + \tau \tfrac{x}{\alpha_n})^2 \lambda \g{n} dF_\lambda \|\yd\|^2. 
\]
Obviously, 
\[ \gtau{0} = \g{n}, \qquad \gtau{1} = \g{n-1}, \]
such that  the measure $\betatau{\tau}$ interpolates between  
$d \betan{n-1}$ and  $d \betan{n}$: 
\begin{align*}  d \betatau{0}  &=  d \betan{n},  \qquad d \betatau{1}  =  d \betan{n-1}. 
 \end{align*} 

Let $\pftau{k}{\tau}(x)$ be the polynomials, 
normalized with $\pftau{k}{\tau}(0)=1$,  and  orthogonal  with respect to 
$d \betatau{\tau}$. 
Hence, \[ \pftau{k}{0}(x) = \pf{k}{n}(x) \qquad \text{ and } \qquad 
\pftau{k}{1}(x) = \pf{k}{n-1}(x).  \] 
As before the orthogonal polynomials $\pftau{k}{\tau}$ have all real distinct roots
in $[0,\|A\|^2]$, and let
us denote them by 
\[  \tirootX{k}{\tau}{i} :=  
 \text{$i$th ordered zero of the polynomial  $\pftau{k}{\tau}(x)$ 
 (as a function of $x$). }
\]

The {\em proof of Lemma~\ref{interlace0}}  is a consequence 
of statement \eqref{lamma0} 
in the 
following lemma, which is based on  Markoff's theorem \cite[Theorem 6.12.1]{Szego}:
\begin{lemma}\label{lemma188}
With the notation above,  we have 
\begin{equation}\label{help3} 
\tirootX{k}{\tau}{i}  \text{ is increasing in $\tau$} \quad \forall 1 \leq i \leq  k, k < n_{bd}. 
\end{equation}
In particular, it follows that  
\begin{equation}\label{lamma0} \rootX{k}{n}{i} <  \rootX{k}{n-1}{i}
\text{ and } \quad 
\rootXk{k}{n}{i} \leq   \rootXk{k}{n-1}{i}.
\end{equation}
\end{lemma} 
\begin{proof}
The result of    \cite[Theorem 6.12.1]{Szego} applies. It states a monotonicity of the 
zeros of orthogonal polynomials with respect to a parameter in the measure under 
a certain condition for the measure. 
Applied to our case, the main condition needed is that 
$\frac{\partial  (\gtau{\tau})^2 }{\partial \tau}/(\gtau{\tau})^2 $
is an increasing function of $\lambda$ in $(0,\|A\|^2)$. This expression 
calculates to $2\frac{\frac{\lambda}{\alpha_n}}{1 +\tau\frac{\lambda}{\alpha_n}}$, 
which is an increasing function in $\lambda$, and hence, the result follows.
(We will use this result in more depth below  in \eqref{lambdaprime}, where also 
the formula for the derivative of the roots with respect to $\tau$ is detailed.)

We note that the result in \cite[Theorem 6.12.1]{Szego} is only stated 
for the case when the measure  is absolutely continuous with respect to 
the Lebesgue measure, but the 
proof shows that we can verbatim adapt it to the present case. 

In the case of  RatCG, we have to distinguish between even and odd steps: 
If $n = 2k$ is even, then the proof is almost identical up to 
the modified definition 
\begin{equation}\label{moddefdef}
\gtau{\tau}:= \g{n}(1+  \tau \tfrac{\lambda}{\alpha_k}) 
= \frac{(1 +\tau \tfrac{\lambda}{\alpha_k})}{\Pi_{i=1}^k (1 +\tfrac{\lambda}{\alpha_i})} \qquad k= \lfloor\frac{n}{2} \rfloor. 
\end{equation} 
In case of $n = 2k+1$, we observe that the polynomials $\pfk{k}{n}$ and $\pfk{k}{n-1}$ 
are orthogonal with respect to the same measure, hence identical and the result is trivial.
\end{proof}

For the proof of  Lemma~\ref{interlace1}, we need some intermediate 
results: Similar to $\pftau{k}{\tau}$ we define an analogue of the polynomials $\we{n-1}$:
For $\tau \in [0,1]$, define the degree-$(n-1)$  (with respect to  $\lambda$) polynomials:
\begin{equation}\label{defmutau} 
\we{}(\lambda,\tau):= \lambda^{-1}   \left(\pftau{n-1}{\tau}(\lambda) -  \pf{n}{n}(\lambda) \right) .
\end{equation}
Then, 
\begin{align*} \we{}(\lambda,0) &=  \lambda^{-1}  \left(\pf{n-1}{n}(\lambda) -  \pf{n}{n}(\lambda) \right), 
\quad \text{ and } \\
 \we{}(\lambda,1) &=  \lambda^{-1}  \left(\pf{n-1}{n-1}(\lambda) -  \pf{n}{n}(\lambda) \right) = \we{n-1}, 
\end{align*}
We prove the following result:
\begin{lemma}\label{lemma18}
Let $n < n_{bd}$.
For all $\tau \in [0,1]$, the polynomial   $\we{}(\lambda,\tau)$ has $n-1$ real and distinct roots. 
The roots depend continuously on $\tau$. 
In particular, $\we{n-1}$ has $n-1$ real roots. 
The same is true for  $\wek{n-1}$
\end{lemma}
\begin{proof}
We note that  $\we{}(\lambda,0)$ is a difference of two consecutive orthogonal polynomials 
with respect to the same measure. It is a well-known result that it has 
$n-1$ real roots that interlace with the roots of $\pf{n-1}{n}$;
see, e.g., \cite[p.~190]{EHN}, \cite[Corollary 2.7]{Hanke}. 
For a given $\tau \in [0,1]$, let $\mu_*$ denote one of its real roots
if it exits. We prove that $\frac{\partial}{\partial \mu} \we{}(\mu_*,\tau) \not = 0$. 
Suppose this is not true. Then $\we{}(\mu,\tau)$ has  at least a double 
root at $\mu_*$ and we can factorize 
$\we{}(\lambda,\tau) = (\lambda - \mu_*)^2 p_{n-3}$. 
Recall that $\pftau{n-1}{\tau}$ and $\pf{n}{n}$ satisfy an orthogonality relation
(with respect to the measures $ (\gtau{\tau})^2 \lambda d F_\lambda(\|\yd\|^2)$
and $ (\g{n})^2 \lambda d F_\lambda(\|\yd\|^2)$, respectively). Hence, 
\begin{align}\label{subsub} \scp{\tfrac{\we{}(\lambda,\tau)}{ (\lambda - \mu_*)^2 }, \pftau{n-1}{\tau} ({\gtau{\tau}})^2 \lambda}  = 0, \qquad 
 \scp{\tfrac{\we{}(\lambda,\tau)}{ (\lambda - \mu_*)^2 }, \pf{n}{n} {\g{n}}^2  \lambda}  = 0, \end{align}
but also 
\[ \scp{\tfrac{\we{}(\lambda,\tau)}{ (\lambda - \mu_*)^2 } \tfrac{\lambda}{\alpha},  \pf{n}{n} {\g{n}}^2  \lambda}  = 0 
\quad \text{ and } \quad  \scp{\tfrac{\we{}(\lambda,\tau)}{ (\lambda - \mu_*)^2 } (\tfrac{\lambda}{\alpha})^2 , \pf{n}{n} {\g{n}}^2  \lambda}  = 0 \]
since the polynomial in the first argument are at most of degree $n-1$. The last three identities combine to 
\[ \scp{\tfrac{\we{}(\lambda,\tau)}{ (\lambda - \mu_*)^2 }  (1 + \tau \tfrac{\lambda}{\alpha})^2 , \pf{n}{n} {\g{n}}^2  \lambda}  = 0. \]
Noting that ${\g{n}}^2 (1 + \tau \frac{\lambda}{\alpha})^2  = ({\gtau{\tau}})^2$ and by subtraction of \eqref{subsub}, 
we find 
\[ \scp{\tfrac{\we{}(\lambda,\tau)}{ (\lambda - \mu_*)^2 } ,(\pftau{n-1}{\tau}-  \pf{n}{n} ) ({\gtau{\tau}})^2 \lambda}  = 0 \]
This reads 
\[ \scp{\frac{(\pftau{n-1}{\tau}-  \pf{n}{n} )^2}{(\lambda -\mu_*)^2},  ({\gtau{\tau}})^2 }  = 0. \]
However, by definiteness this means that 
$\pftau{n-1}{\tau}=  \pf{n}{n}$ which is a contradiction. 
Thus, $\frac{d}{dx} \we{}(\mu,\tau)$ is always nonzero at at root, and by starting 
at the $n-1$ roots for $\tau =0$ and using the implicit function theorem, we obtain 
$n-1$ real root functions for $\tau \in [0,1]$, which, as a consequence of the implicit function 
theorem, are continuous in $\tau$. 

The result for the RatCG case follows in the same way. In case of $n = 2k$, we use the 
modification \eqref{moddefdef} and proceed in the exactly same way. In case of 
$n = 2k+1$, the polynomial $\wek{}$ corresponds to  the difference of two successive polynomials 
for the same measure, and it is well-known that in this case we have $n-1$ real roots
(compare the polynomial $u_k$ in \cite[p.~190]{EHN}). 
\end{proof}

{\em Proof of Lemma~\ref{interlace1}:}
We first prove \eqref{lalala}.
Recall the definition of $\muone{n-1}$ in \eqref{def:muone}. 
Denote by $\muoneti{n-1}(\tau)$ the smallest root of $\we{}(\lambda,\tau)$
in \eqref{defmutau}
(such that $\muoneti{n-1}(1) = \muone{n-1}$). 
 For the case $\tau = 0$, we use the above 
cited interlacing result (recall \cite[p.~190]{EHN}, \cite[Corollary 2.7]{Hanke}) that 
\begin{equation}\label{help111} \rooti{n-1}{n} < \muoneti{n-1}(0) <  
 \rootX{n-1}{n}{2}. 
\end{equation}
Suppose that 
\begin{equation}\label{contra} \muoneti{n-1}(1) < \rooti{n-1}{n}, \end{equation}
and we show that this leads to a contradiction. Using \eqref{help111} and 
\eqref{contra} and the fact that 
$ \muoneti{n-1}$ is continuous in $\tau$, we apply the intermediate value theorem 
to conclude the existence of a $\tau^*$ such that 
$ \muoneti{n-1}(\tau^*) = \rooti{n-1}{n}$.  
At this root we have by  definition of  $\we{}(\lambda,\tau)$ 
(and $\we{}( \muoneti{n-1}(\tau^*),\tau^*)=0$) that 
\[ \pftau{n-1}{\tau^*}(\muone{n-1}(\tau^*)) = \pf{n}{n}(\muone{n-1}(\tau^*)) = 
 \pf{n}{n}(\rooti{n-1}{n}).
\]
Again by the interlacing theorem, it follows that  $\rooti{n-1}{n} \in (\rooti{n}{n},\rootX{n}{n}{2})$,
and hence $ \pf{n}{n}(\rooti{n-1}{n})  <0$. Thus
$ \pftau{n-1}{\tau^*}(\muone{n-1}(\tau^*)) <0$, which means that the first zero of $\pftau{n-1}{\tau^*}$
(its notation is $\tirootX{n-1}{\tau^*}{1}$)
occurs before 
$\muoneti{n-1}(\tau^*)$. This and the monotonicity of $
\tirootX{n-1}{\tau}{1}$ in the $\tau$-variable
lets us conclude that 
\[  \muone{n-1}(\tau^*) > 
\tirootX{n-1}{\tau^*}{1} > 
\tirootX{n-1}{0}{1} = \rooti{n-1}{n}. \]
This is a contradiction to $ \muone{n-1}(\tau^*) = \rooti{n-1}{n}$.

Regarding \eqref{anoy}, we show the following: 
\[\text{ if } \muone{n-1} < \rootX{n-1}{n-1}{1}, \quad \text{ then } \quad   \muone{n-1} \geq \lambda_{2,n}^n \] 
Suppose this is not the case. Then, 
as we already know that $\muone{n-1} \geq \rooti{n-1}{n} >\rooti{n}{n}$, it follows  that 
$\muone{n-1} \in (\rooti{n}{n}, \rootX{n}{n}{2})$. 
However,   $\pf{n}{n}$ is negative at $\muone{n-1}$, but 
by $\muone{n-1} < \rootX{n-1}{n-1}{1}$,
the polynomial $\pf{n-1}{n-1}$ is strictly positive at $\muone{n-1}$. 
Thus,  $\muone{n-1}$ cannot be a root of $\we{n-1}= \lambda^{-1}(\pf{n-1}{n-1}- \pf{n}{n})$,
which is a contradiction.  

The odd case $n = 2k$  for RatCG can be proven in exactly the same way with the analogous modified 
definitions. In case $n = 2k +1$, the result for the CGNE-case, \cite[p.~190]{EHN}
states that $\rootXk{n-1}{n-1}{1} < \muonek{n-1}$,  and since the measures 
for $n$ and $n-1$ are identical, we have $\rootXk{n-1}{n-1}{1}= \rootXk{n-1}{n}{1}$, proving \eqref{lalala}
for this case. The identity  \eqref{anoy} is a consequence of 
the interlacing for the CGNE case \cite[p.~190]{EHN}:
$\rootXk{n-1}{n-1}{1} < \muone{n-1} < \rootXk{n-1}{n-1}{2}$ and noting that 
$\rootXk{n-1}{n-1}{2} < \rootXk{n}{n-1}{2} = \rootXk{n}{n}{2}$.
\qed

 \smallskip

We now proceed to the {\em Proof of Lemma~\ref{interlace2}}. 
We again use  Markoff's theorem
\cite[Theorem 6.12.1]{Szego}:
Taking hte derivative with respect to $\tau$  (the notation is a prime $'$)
gives 
\[ d (\betatau{\tau})^\prime =  2 (1 +\tau \tfrac{x}{\alpha_n}) \tfrac{x}{\alpha_n}d \betan{n}  =
2 \frac{\tfrac{x}{\alpha_n}}{1+ \tau \tfrac{x}{\alpha_n}}  d \betatau{\tau} . \]

Let $\lambda_*(\tau)$, $\tau \in [0,1]$ be the path of the smallest root of $\pftau{n-1}{\tau}$. 
The cited theorem of Markoff states that 
the roots are differentiable with respect to $\tau$ 
and a formula for the derivative can be established. In our case this leads to  
(\cite[Equations~(6.12.4) and (3.4.6)]{Szego}:
\begin{equation}\label{lambdaprime} \lambda_*'(\tau) = 
\frac{\int_0^{\|A\|^2} \frac{(\pftau{n-1}{\tau})^2}{\lambda-\lambda_*} 
\frac{d (\betatau{\tau})'}{d\betatau{\tau}} 
d \betatau{\tau} }{\int_0^{\|A\|^2}  \frac{(\pftau{n-1}{\tau})^2}{(\lambda-\lambda_*)^2} d \betatau{\tau} } 
= 
\frac{2}{\alpha_n} \frac{\int_0^{\|A\|^2}  \frac{(\pftau{n-1}{\tau})^2}{\lambda-\lambda_*} 
\tfrac{\lambda}{1+ \tau \tfrac{\lambda}{\alpha_n}} 
d \betatau{\tau} }{\int_0^{\|A\|^2}  \frac{(\pftau{n-1}{\tau})^2}{(\lambda-\lambda_*)^2} d \betatau{\tau} } 
.\end{equation}
Obviously, this  again  verifies the monotonicity of the roots with respect to $\tau$ claimed in Lemma~\ref{lemma188}.
By the orthogonality of $\pftau{n-1}{\tau}$ and the fact that 
$\frac{\pftau{n-1}{\tau}}{\lambda-\lambda_*} \in \P^{n-2}$, we have that 
\begin{equation}\label{ooo} \int_0^{\|A\|^2}  \frac{(\pftau{n-1}{\tau})^2}{\lambda-\lambda_*}  d\betatau{\tau} = 0.
\end{equation}
We estimate the numerator in \eqref{lambdaprime} in two ways: Dropping the term where $\lambda-\lambda_* <0$
gives, using $\lambda \leq \|A\|^2$ and  
$1+ \tau \tfrac{\lambda}{\alpha_n} \geq 1$, 
\begin{align*} 
\frac{2}{\alpha_n}&\int_0^{\|A\|^2}  \frac{(\pftau{n-1}{\tau})^2}{\lambda-\lambda_*} 
\frac{\lambda}{1+ \tau \tfrac{\lambda}{\alpha_n}}  d \betatau{\tau}\leq 
\frac{2}{\alpha_n}\int_{\lambda_*}^{\|A\|^2} \frac{(\pftau{n-1}{\tau})^2}{\lambda-\lambda_*} 
\frac{\lambda}{1+ \tau \tfrac{\lambda}{\alpha_n}}  d \betatau{\tau}\\
& \leq \frac{2 \|A\|^2}{\alpha_n}
\int_{\lambda_*}^{\|A\|^2} \frac{(\pftau{n-1}{\tau})^2}{\lambda-\lambda_*} d \betatau{\tau} 
=_{\ref{ooo}} 
 \frac{2 \|A\|^2}{\alpha_n}
\int_{0}^{\lambda_*} \frac{(\pftau{n-1}{\tau})^2}{\lambda_*-\lambda} d \betatau{\tau} \\
&=  
 \frac{2 \|A\|^2}{\alpha_n}
\int_{0}^{\lambda_*} \frac{(\pftau{n-1}{\tau})^2}{(\lambda_*-\lambda)^2} 
(\lambda_*-\lambda) d \betatau{\tau} 
\leq 
 \frac{2 \|A\|^2}{\alpha_n}\lambda_* 
 \int_{0}^{\lambda_*}  \frac{(\pftau{n-1}{\tau})^2}{(\lambda-\lambda_*)^2} d \betatau{\tau} .
\end{align*}
On the other hand using 
$\frac{\frac{\lambda}{\alpha_n}}{(1+ \tau \tfrac{\lambda}{\alpha_n})}   = 
\frac{1}{\tau} \left(1 - \frac{1}{(1+ \tau \tfrac{\lambda}{\alpha_n})}\right)$
and again \eqref{ooo}, we estimate
\begin{align*} 
\frac{2}{\alpha_n}&\int_0^{\|A\|^2}  \frac{(\pftau{n-1}{\tau})^2}{\lambda-\lambda_*} 
\frac{\lambda}{1+ \tau \tfrac{\lambda}{\alpha_n}}  d\betatau{\tau} =_{\ref{ooo}} 
-\frac{2}{\tau}\int_0^{\|A\|^2}  \frac{(\pftau{n-1}{\tau})^2}{\lambda-\lambda_*} 
\frac{1}{1+ \tau \tfrac{\lambda}{\alpha_n}}  d\betatau{\tau} \\
&  \leq 
\frac{2}{\tau}\int_0^{\lambda_*} \frac{(\pftau{n-1}{\tau})^2}{\lambda_* -\lambda} 
\frac{1}{1+ \tau \tfrac{\lambda}{\alpha_n}}  d\betatau{\tau}  \leq 
\frac{2}{\tau}\int_0^{\lambda_*} \frac{(\pftau{n-1}{\tau})^2}{(\lambda_* -\lambda)^2}
(\lambda_* -\lambda)  d\betatau{\tau}  \\
& \leq 
\frac{2\lambda_*}{\tau} \int_0^{\lambda_*}  \frac{(\pftau{n-1}{\tau})^2}{(\lambda_* -\lambda)^2} d\betatau{\tau} .
\end{align*}
The integral $\int_0^{\lambda_*}  \frac{(\pftau{n-1}{\tau})^2}{(\lambda_* -\lambda)^2}$ is bounded 
by the denominator in  \eqref{lambdaprime}. 
Thus, we obtain 
\[ \lambda_*'(\tau)  \leq 
2 \min\left\{\frac{ \|A\|^2}{\alpha_n},
\frac{1}{\tau} \right\} \lambda_*. \]
Dividing by $\lambda_*$ and integrating $\int_0^1 d \tau$ gives, 
in case $\frac{\alpha_n}{ \|A\|^2} <1$, 
\begin{align}  \log(\lambda_*(1)) - \log(\lambda_*(0)) &\leq 
2 \int_0^{\frac{\alpha_n}{ \|A\|^2}} 
\tfrac{\|A\|^2}{\alpha_n} d\tau + 
2 \int_{\tfrac{\alpha_n}{ \|A\|^2}}^1 \tau^{-1} d\tau \label{secondpart}\\
& = 
2 + 2 \log(1) - 2 \log(\tfrac{\alpha_n}{ \|A\|^2}). \nonumber 
\end{align}
Thus, 
\[ \lambda_*(1) \leq \lambda_*(0) e^{2} \left(\tfrac{\|A\|^2}{\alpha_n}\right)^2. \]
 In case $\frac{\alpha_n}{ \|A\|^2} \geq 1$, the second part of the integral 
 in \eqref{secondpart} 
 vanishes, and we have 
 \[ \log(\lambda_*(1)) - \log(\lambda_*(0)) \leq 
2 \int_0^{1} 
\tfrac{\|A\|^2}{\alpha_n} d\tau \leq 2, 
\] 
 yielding $\lambda_*(1) \leq \lambda_*(0) e^{2} . $
Thus 
 \[ \lambda_*(1) \leq \lambda_*(0) e^{2} \max\left\{\left(\tfrac{\|A\|^2}{\alpha_n}\right)^2,1\right\}. \]
The definition of $\lambda_*(1)$,$\lambda_*(0)$ yields  Lemma~\ref{interlace2}.

The case for RatCG is again identical in the case $n = 2k$. In case $n = 2k-1$, we have that 
$\rootXk{n-1}{n-1}{1} = \rootXk{n-1}{n}{1}$ such that the assertion is trivially satisfied. 
\qed

\end{document}